\documentclass{amsart}

\usepackage{amssymb, amsmath, amsfonts, amscd,stackengine}
\usepackage[raiselinks=false,colorlinks=true,citecolor=blue,urlcolor=blue,
linkcolor=blue,bookmarksopen=true,pdftex]{hyperref}
\usepackage[dvipsnames]{xcolor}
\usepackage{amsmath,tikz-cd}

\usepackage[english]{babel}
\usepackage[utf8]{inputenc}
\usepackage{comment}
\usepackage{setspace}
\usepackage{hyperref}
\usepackage{xcolor}
\usepackage{lipsum}

\usepackage{enumerate}
\usepackage[mathscr]{eucal}
\newtheorem{theorem}{Theorem}[section]
\newtheorem{proposition}[theorem]{Proposition}
\newtheorem{lemma}[theorem]{Lemma}
\newtheorem{remark}[theorem]{Remark}
\newtheorem{definition}[theorem]{Definition}
\newtheorem{example}[theorem]{Example}
\newtheorem{corollary}[theorem]{Corollary}

\DeclareMathOperator*\dtimes{\times\mkern-11.00mu
  \ensurestackMath{\raisebox{-0.90pt}{\stackanchor[-3.50pt]{\times}{\times}}}
  \mkern-10.90mu\times}
\usepackage{graphicx} 
\begin{document}
\setstretch{1.4}
\title[Fundamental groups of perforated surfaces]{On the fundamental groups of perforated surfaces}
\author[K. Gulati]{Khushbu Gulati}
\author[P. Sankaran]{Parameswaran Sankaran}
\address{Chennai Mathematical Institute, H1 SIPCOT IT Park, Siruseri, Kelambakkam 603103, Tamil Nadu, India}
\email{khushbu@cmi.ac.in}
\email{sankaran@cmi.ac.in}
\date{}
\begin{abstract} 
A {\em perforated surface} is the complement $\mathring\Sigma:=\Sigma\setminus A$ of a countable dense subset $A$ in a connected paracompact surface $\Sigma$.  It is known that the topological type of $\Sigma\setminus A$ 
is independent of the choice of $A$.  Any perforated surface is one-dimensional, connected, locally path connected, and is not semi-locally simply connected at any of its points. 

In this paper we obtain a classification theorem for perforated surfaces, using the classification theorem for surfaces. We show that any connected covering of a perforated surface $\mathring \Sigma$ arises from a covering of a surface $\Sigma'$ such that $\mathring\Sigma\cong \mathring\Sigma'$. 
We show that the fundamental group of perforated surfaces are large.  For example, 
the countable free $\sigma$-product $\dtimes_{\mathbb N}\pi_1(\mathcal S)$ of $\pi_1(\mathcal S)$, where $\mathcal S$ stands for the Sierpiński curve, and the 
free product  $*_{\mathfrak c}\pi_1(\mathring \Sigma)$ embed in $\pi_1(\mathring\Sigma)$.  
   
We also show that the fundamental groups of $\mathring \Sigma$, the Sierpiński curve and the Menger curve are not Hopfian.  
\end{abstract}
\maketitle

\section{Introduction}\label{intro}

Let  $\Sigma$ be any connected paracompact surface without boundary.  Let $A$ be any countable dense subset of $\Sigma$.  We 
denote by $\mathring \Sigma$ the path connected space $\Sigma\setminus A.$  Bennett \cite{bennett} had shown that any connected metrizable manifold $M$ is countable dense homogeneous, that is, if $A_0,A_1$ are countable dense subspaces of $M$, there exists a homeomorphism of $M$ that maps $A_0$ onto $A_1$.  In view of this, the homeomorphism type of $\mathring \Sigma=\Sigma\setminus A$ is independent of the choice of the countable dense subset $A$.   We refer to $\mathring \Sigma$
as a {\em perforated surface} or a {\em perforation} of $\Sigma$.  
 An important special case is the space $\mathbb R^2\setminus \mathbb Q^2,$ which we shall denote by $\mathcal N.$  It is readily seen that there are non-homeomorphic surfaces having 
homeomorphic perforations.  For example, $\mathring {\mathbb S}^2\cong \mathcal N.$
 
Any perforated surface has the Lebesgue covering dimension equal to one.  
No point of a perforated surface has a neighbourhood which is semi-locally simply connected.  Such spaces are known as {\em wild spaces}. By a deep result of 
Eda, one-dimensional, connected, locally path connected, metrizable wild spaces are determined, up to homeomorphism, by their fundamental groups.  See \cite{eda-2002}.  We point out below some 
properties of the fundamental groups of perforated surfaces that are well-known or are easy consequences of well-known results. \\
(a) 
The fundamental group $\pi_1(\mathring\Sigma)$ is locally free and embeds in an inverse limit of free groups. This follows from the fact that $\mathring \Sigma$ embeds in the Menger curve $\mathcal M$ and that $\pi_1(\mathcal M)$ embeds in an inverse limit of finite rank free groups. See \cite[Theorem 5.11]{cannon-conner}, and also \cite{eda-kawamura}.  Therefore $\pi_1(\mathring\Sigma)$ is also residually finite.  The space $\mathring\Sigma$ is an Eilenberg-Mac Lane space $K(\pi,1)$ and $H_j(\mathring\Sigma;\mathbb Z)=0$ if $j>1$. (See \cite{curtis-fort-1957}, \cite{curtis-fort}, \cite{cannon-conner-zastrow}.) \\
(b) Any free abelian quotient group of $\pi_1(\mathring\Sigma)$ has countable rank.  Hence any non-abelian free quotient group of $\pi_1(\mathring \Sigma)$ has countable rank.  This is a special case of a theorem due to Cannon and Conner \cite[Theorem 5.1]{cannon-conner}.\\
(c) Any subspace $X$ of $\Sigma$ with empty interior can be embedded in $\mathring \Sigma$. It follows, from \cite[Corollary 3.3]{cannon-conner}, that $\pi_1(X)$ embeds in $\pi_1(\mathring \Sigma)$.\\
(d) The homeomorphism group $Homeo(\mathring \Sigma,x_0)$ of $\mathring\Sigma $ fixing a base point $x_0$ operates {\em faithfully} via automorphisms 
on $\pi_1(\mathring \Sigma, x_0)$. \\

 In this paper, we obtain a classification theorem for perforated surfaces, 
 using the end space $E(\Sigma)$ of $\Sigma$ (in the sense of Freudenthal), in the spirit of the  
classification of connected metrizable surfaces due to Kerékjártó \cite[Chapter 5]{kerekjarto} and Richards \cite{richards}.   Denote by $\mathcal E(\Sigma)$ the Freudenthal's  compactification $\Sigma \cup E(\Sigma).$
We obtain the following result in \S2. See also Corollary \ref{classification-nested-triples}.

\begin{theorem}\label{homeo-no-planar-end}
Suppose that $\Sigma_1$ and $\Sigma_2$ are connected surfaces with compact 
boundaries (possibly empty).  Suppose $A_1,A_2$ are countable dense subsets of 
$\Sigma_1,\Sigma_2$ respectively such that $A_1\cap \partial \Sigma_1=\emptyset=A_2\cap\partial \Sigma_2.$  Then, any homeomorphism $f:\mathring \Sigma_1\to \mathring \Sigma_2$, extends to a homeomorphism  
of the Freudenthal compactifications $\mathcal E(f):\mathcal E(\Sigma_1)\to \mathcal E(\Sigma_2).$  If neither of the surfaces $\Sigma_1, \Sigma_2$ has a planar end, then $f$ extends to a homeomorphism 
$\tilde f:\Sigma_1\to \Sigma_2.$
\end{theorem}

Also, we show that any connected Hausdorff space $X$ with a countable open cover $\mathcal U=\{U\}$ such that $U\cong \mathcal N$ and $\partial U\cong \mathbb S^1$ is a perforated surface. 
The proof is essentially an adaptation of the proof of T. Rado's theorem 
on the triagulability of second countable surfaces, as given in \cite[Chapter I, \S 45-46]{ahlfors-sario}. 

Although a perforated surface is a wild space, we show that it admits plenty of coverings. (See Proposition \ref{coverings-countable-deck-group} and Corollary \ref{c-many-coverings}.)  We show, in Theorem \ref{coverings-are-countable}, that 
any connected covering of $\mathring \Sigma$ has degree either finite or 
countably infinite.  

Besides being interesting objects of study from a purely topological point of view, 
the family of fundamental groups of perforated surfaces gives an interesting class of groups.  We show that these groups are rather large and seem to have an intricate normal subgroup structure. (See \S \ref{normal-subgroups}.)  We obtain the following embedding result.



Let $\{G_\alpha\}_{\alpha\in J}$ be an indexed family of groups.  One has the notion of the free $\sigma$-product $\dtimes_{\alpha\in J}^\sigma G_\alpha$. (The symbol $\sigma$ alludes to `countable'.)
When $G_\alpha=G~\forall \alpha\in J$, we denote the free $\sigma$-product of $\{G_\alpha\}$ by $\dtimes_J^\sigma G$.  When $J$ is countably infinite, we omit $\sigma$ from the notation; in this case  
$\dtimes_{\alpha\in J}G_\alpha$ is the same as the complete free product of $\{G_\alpha\}$ in the sense of Higman \cite{higman}. 
We refer the reader to \cite{eda-ja} for detailed exposition on the free $\sigma$-product of groups. 

For each $\alpha\in J$, let $(X_\alpha,x_\alpha)$ be a based topological space which has a countable basis at $x_\alpha$ and is locally simply connected at $x_\alpha$.  
Let $X:=\bigvee_{\alpha\in J} X_\alpha$ be their `weak join',  regarded as a subspace of the product space $\prod_{\alpha\in J}X_\alpha$ with base point $x_0:=(x_\alpha, \alpha\in J)$.
Morgan and Morrison \cite{morgan-morrison} described the fundamental group $\pi_1(\bigvee_{\alpha\in J}X_\alpha, x_0)$ as a subgroup of the inverse limit of an inverse system of groups, when the indexing set $J$ is countable.   (This result was first asserted by H. B. Griffiths \cite{griffiths}.) 
It turns out that, for any $J$, $\pi_1(X,x)\cong \dtimes_{\alpha\in J}^\sigma \pi_1(X_\alpha,x_\alpha)$. See \cite[Theorem A1]{eda-ja}.  
 
Let $F_\mathfrak c$ be a free group of rank $\mathfrak c:=|\mathbb R|$.  We obtain the following embedding theorem. 

\begin{theorem}
 Let $X\in \{\mathcal H, \mathcal S,\mathcal T\}$. (a) Then $\pi_1(X)$ embeds in $\pi_1(\mathring\Sigma)$ as a retract.\\
(b) The free product $(\dtimes_\mathbb N\pi_1(\mathcal N)) *\pi_1(\mathring \Sigma)$  
embeds in $\pi_1(\mathring\Sigma)$.\\
(c) The free product  
$*_\mathfrak c \pi_1(\mathring\Sigma)$ 
of $\mathfrak c$ many copies of $\pi_1(\mathring\Sigma)$ embeds in $\pi_1(\mathring\Sigma)$.\\
\end{theorem}
For part (a), see Example \ref{fundamentalgroups-as-retracts}.
Parts (b) and (c) are proved in 
Theorem \ref{freeproduct-of-pi1-sigma-perforated}.  
 
It is easily seen that there are examples of such spaces (such as the Hawaiian earring) whose fundamental groups are not Hopfian.
 We do not know if $\pi_1(X)$ is non-Hopfian for {\em any} wild one-dimensional separable, connected, locally connected metric space $X$.   We have the following theorem 

\begin{theorem}
(i) The group $\pi_1(\mathring\Sigma,x_0)$ is not Hopfian for any surface $\Sigma.$\\
(ii) The group $\pi_1(X,x_0)$ is non-Hopfian if $X$ is the Sierpiński curve, the Sierpiński gasket, or the Menger curve. 
\end{theorem}

\subsection{Notations} 
If $A\subset X$, $\overline A$ denotes the closure of $A$ in $X$, $\partial_X A$ (more briefly, $\partial A$) denotes the set of boundary points of $A$ in $X$, i.e., $\partial A=\overline{A}\cap \overline{X\setminus A}.$  If $M\subset X$ is a 
manifold-with-boundary, $\partial M$ denotes the manifold-boundary of $M$.  
For a closed space $A\subset X$, $X/A$ denotes the quotient 
space obtained by collapsing $A$ to a point.

If $X, Y$ are path connected topological spaces with base points $x_0\in X, y_0\in Y$, $X\vee Y$ denotes the subspace 
$\{(x,y)\in X\times Y\mid x=x_0~\textrm{or~} y=y_0\}.$  More generally, if $(X_\alpha,x_\alpha)_{\alpha\in J}$ is a family of pointed topological spaces, $\bigvee_{\alpha\in J}X_\alpha$ denotes the subspace 
$\{(y_\alpha)_{\alpha\in J}\in X:=\prod_{\alpha\in J} X_\alpha\mid y_\alpha\ne x_\alpha\textrm{~for~at most~one~} \alpha\in J\}$, where $X$ is given the product topology. 

We often suppress the base point $x_0$ and write $\pi_1(X)$ to denote $\pi_1(X,x_0).$   If $X$ is a metric space, $B_r(x)\subset X$ (resp. $ D_r(x)\subset X$) denote the open ball (resp. closed ball) centred at $x\in X$ and radius $r$. We shall denote by $S_r(x)$ the `sphere' $\{y\in X\mid d(x,y)=r\}$. 

We denote by 
$\mathcal H, \mathcal N,\mathcal S,\mathcal T,\mathcal M$ the Hawaiian earring space, the space $\mathbb R^2\setminus \mathbb Q^2,$ the Sierpiński curve, the Sierpiński gasket, and the Menger curve respectively. 

Let $\Sigma $ be any surface (with or without boundary) and $A$, a countable dense subset of $\Sigma$.  We denote by $\mathring \Sigma$ the space $\Sigma\setminus A$.  When $\partial \Sigma\ne \emptyset,$
we assume that $A\cap \partial \Sigma=\emptyset.$  

By a \emph{subsurface} $S$ of a surface $\Sigma$, we mean an embedded compact 
surface-with-boundary where $\partial_\Sigma S$ is a union of finitely many pairwise disjoint circles.  A subsurface is also called a bordered subsurface. When $\Sigma$ is connected and $\partial S=\emptyset$, then $S=\Sigma$.


Let $\gamma:[a,b]\to X$ and $\sigma:[a,b]\to X$ be paths where $\sigma(a)=\gamma(b)$. As usual $\gamma\cdot \sigma:[a, b]\to X$ is the concatenation $\gamma\cdot\sigma$. Explicitly, 
\[\gamma.\sigma(t)=\begin{cases}
\gamma(a+2t(b-a)), & ~\textrm{if~}0\le t\le 1/2,\\
\sigma(a+(2t-1)(b-a)), &~\textrm{if~} 1/2\le t\le 1.\\
\end{cases}
\]
Since we are 
concerned only with homotopy classes of loops, we shall identify 
two paths $\xi:[a,b]\to X$ and $\eta:[a,b]\to X$ and write $\xi=\eta$,  
if one can be obtained from the other by composing with an order preserving self-homeomorphism of $[a,b].$  This makes the    
concatenation operation associative so that we have $(\xi\cdot\lambda)\cdot\zeta=\xi\cdot(\lambda\cdot\zeta)$ (whenever they are defined) and either of them is denoted as $\xi\cdot\lambda\cdot \zeta.$  We write 
$\sigma\simeq \lambda$ 
if $\sigma,\lambda:[a,b]\to X$ are homotopic relative to $\{a,b\}$. 



Let $X$ be a $1$-dimensional path connected Hausdorff topological space. 
An {\em arc} in  $X$ 
is an embedding $\sigma: [a,b]\to X.$  
A non-constant loop $\gamma: I\to X$ based at $x_0\in X$ is said to be {\em reduced} if there is no subinterval 
$[a,b]\subset I$  with $a<b$ such that $\gamma|_{[a,b]}$ is a null-homotopic loop.   
The constant loop is considered, by convention, to be reduced. 

\section{Classification of perforated surfaces}
We recall the notion of the end space of a surface and use it to obtain a classification theorem  
for perforated surfaces. As a consequence we shall show that there is a family 
of groups $\{\pi_1(\mathring \Sigma_\alpha)\}_{\alpha\in J}$ of cardinality $2^{\aleph_0}$ which are pairwise non-isomorphic. 

\subsection{Genus and the orientation class}
Let $\Sigma$ be a connected compact surface with $k$ boundary components, $k\ge 0$.  If $\Sigma$ is orientable of genus $g$, then 
its Euler characteristic equals $\chi(\Sigma)=
2-2g-k$ and the surface is homeomorphic to the $2$-sphere with $g$ handles in which interiors of $k$ embedded disks have been removed.
In the case $\Sigma$ is non-orientable of genus $g$ then 
$\chi(\Sigma)=2-g-k$ and the surface $\Sigma$ is homeomorphic to the $2$-sphere with $g$ cross-caps in which interiors of $k$ embedded disks have been removed.  

Let $\Sigma$ be a non-compact connected separable surface. Then $\Sigma$ is {\em planar} if it can be embedded in $\mathbb R^2.$   It is said to be of {\em genus} $g\in \mathbb N$ if there exists a compact (orientable or non-orientable) surface-with-boundary $S\subset \Sigma$ 
of genus $g$ such that every component of $\Sigma\setminus S$ is planar. 

The surface $\Sigma$ is of {\em infinite genus} (resp. {\em infinitely non-orientable}) if, for any compact surface-with-boundary $S\subset \Sigma$, the surface $\Sigma\setminus S$ is not planar (resp.  
non-orientable).  If $\Sigma$ is non-orientable but $\Sigma\setminus S$
is orientable for a compact surface-with-boundary $S$, then $\Sigma$ is 
{\em finitely non-orientable}. 

Suppose that $\Sigma$ is finitely non-orientable.   Since $P^2\# P^2\#P^2\cong P^2\# T,$ it can be seen that exactly one of the following alternative holds: $\Sigma \cong P^2\#\Sigma_0$ or $\Sigma\cong P^2\# P^2\# \Sigma_0$ for an orientable surface $\Sigma_0$. 
In the former case $\Sigma$ is of {\em odd} type, and in the latter, it is of {\em even} type. Thus, when the genus is infinite, there are four {\em orientability classes}: one orientable class $\mathrm{O},$ 
and three non-orientable classes (collectively denoted $\textrm{NO}$), namely, 
infinitely non-orientable $\textrm{NO}_\infty$, (finitely) non-orientable of odd 
and even types, denoted $\textrm{NO}_\textrm{o}$ and $\textrm{NO}_\textrm{e}$ respectively. 

\subsection{End spaces and the classification of surfaces} 
We recall here the Freudenthal's end space $E(\Sigma)$ and the end compactification $\mathcal E(\Sigma)$ of a surface $\Sigma$.  (Although the construction is very general and valid for any locally compact Hausdorff second countable topological space, we consider only the case of surfaces.) When $\Sigma$ is compact, its end space is empty and $\mathcal E(\Sigma)=\Sigma.$

Assume that $\Sigma$ is a non-compact connected surface.
Let $\mathcal M$ be the set of all sequences $\underline M=\{M_n\}$ of embedded submanifolds of $\Sigma$ such that:\\ (i) each $M_n$ is a connected component 
of $\Sigma\setminus S_n$ where $S_n\subset \Sigma$ is a compact subsurface (with boundary) of $\Sigma, $\\
(ii) $M_1\supset M_2\supset \cdots$, and,\\   
(iii) if $C\subset \Sigma$ is compact, then $C\cap M_k=\emptyset$ for sufficiently large $k$.

Two sequences  $\underline M=\{M_n\}, \underline {M'}=\{M'_n\}\in \mathcal M$ 
are equivalent or {\em determine the same end} if 
there exist positive integers $p,q$ such that $M_p\subset M'_m$ for some $m=m(p)$ and $M'_q\subset M_n$ for some $n=n(q).$  This is an equivalence relation and the equivalence class 
of $\underline{M}$ is an {\em end} of $\Sigma$ denoted 
$\epsilon(\underline{M})$.  The set of all ends of $\Sigma$ is denoted by $E(\Sigma)$.  

The set $E(\Sigma)$ has a topology with respect to which it is totally disconnected, compact, Hausdorff, and second countable. 
Hence $E(\Sigma)$ is homeomorphic to a closed subspace of the Cantor set $\mathcal C.$  The topology on $E(\Sigma)$ may be described as follows: Let $\mathcal K$ denote the collection of all compact subsurfaces of $\Sigma$. For $\underline M\in \mathcal M, S\in \mathcal K$,  denote by $V(\epsilon(\underline M),S)\subset E(\Sigma)$ the set of all ends $\epsilon (\underline{M'})$ where $M'_n$ and $M_n$ are contained in the same path component of $\Sigma\setminus S$ for sufficiently large $n$. 
A basis of $E(\Sigma)$ is 
$\{V(\epsilon(\underline {M}),S)\mid S\in \mathcal K, \epsilon(\underline M)\in E(\Sigma)\}$.

An end $\epsilon=\epsilon(\underline{M})\in E(\Sigma)$ is {\em planar} (resp. {\em non-orientable}) if $M_k$ is planar for some $k$ (resp. {\em non-orientable for every $k$}). Denote by $  E_{\textrm{np}}(\Sigma)$ 
(resp. $ E_{\textrm{no}}(\Sigma)$) the set of all non-planar (resp. non-orientable) ends of $\Sigma$.  
The sets $E_{\textrm{np}}(\Sigma)$ and $ E_{\textrm{no}}(\Sigma)$ are closed in $E(\Sigma)$
and we have the inclusions   
$E(\Sigma)\supset E_{\textrm{np}}(\Sigma) \supset E_{\textrm{no}}(\Sigma)$.
Then $E_\textrm{p}(\Sigma):=E(\Sigma)\setminus E_{\textrm{np}}(\Sigma)$ is  the set of all planar ends of $\Sigma.$  Note that $E_{\textrm{no}}(\Sigma)\ne \emptyset$ if and only if $\Sigma$ is infinitely non-orientable and that $E_{\textrm{np}}= \emptyset$ if and only if $\Sigma\setminus S\subset \mathbb R^2$ for some compact subsurface $S$.  Of course, $E(\Sigma)=\emptyset$ if and only if $\Sigma$ is compact.  The relation between the orientation class of $\Sigma$ and $E(\Sigma), E_{\textrm{np}}(\Sigma), E_{\textrm{no}}(\Sigma)$ depending on the genus $g(\Sigma)$ is given in Table 1.

\[
\begin{array}{|c||c|c|}
\hline
 g=g(\Sigma) &  \textrm { Orientation} &\textrm{End space}\\
&\textrm{ class }~\Omega & \\
\hline\hline
   g<\infty  & \textrm{O, NO} & E_{\textrm{np}}=\emptyset \\
   \hline
  g=\infty   & \textrm{O},\textrm{NO}_\mathrm{e}, \textrm{NO}_{\mathrm{o}} & E_{\textrm{no}}=\emptyset\\
  & & E_{\mathrm{np}}\ne \emptyset\\
\hline 
g=\infty & \textrm{NO}_\infty & E_{\mathrm{no}}\ne \emptyset\\
\hline
\end{array}
\]
\begin{center}\label{table-genus-orient-ends}
Table 1:  The orientation class and the ends of a surface $\Sigma$ of genus $g$.
\end{center}

Suppose that $E_i\supset E'_i\supset  E_i'', i=0,1,$ are nested closed subspaces of the Cantor set $\mathcal C$.  The triad of spaces $ (E_0,E_0',E_0'')\to (E_1,E_1', E_1'') $ are of the same homeomorphism type if 
there exists a homeomorphism $f:E_0\to E_1$ such that $f(E_0')=E_1', f(E_0'')=E_1''$.

We recall below the following theorem.  

\begin{theorem}\label{classification-surfaces}
{\em (Kerékjártó's classification theorem \cite{kerekjarto},  \cite{richards})} (i)  Any connected separable surface $\Sigma$ is determined up to homeomorphism by its genus, orientation class, and the homeomorphism type of the triple $(E(\Sigma),E_{\mathrm{np}}(\Sigma), E_{\mathrm{no}}(\Sigma)).$ \\
(ii) Conversely, given any nested triple $E\supset E'\supset E''$ of closed subspaces of the Cantor set $\mathcal C$, there exists a surface $\Sigma$ such that $\varepsilon(\Sigma):=(E(\Sigma),E_{\mathrm{np}}(\Sigma),E_{\mathrm{no}}(\Sigma))\cong(E,E',E'')$. 
\hfill $\Box$
\end{theorem}

The set $\mathcal E(\Sigma)=\Sigma \cup E(\Sigma)$ has a topology which makes it a 
compact Hausdorff space in which $\Sigma$ is a dense open set so that 
$\mathcal E(\Sigma)$ is a compactification of $\Sigma$, well-known as 
the Freudenthal compactification of $\Sigma$. The induced subspace topology on $E(\Sigma)$ is the same as the end space topology.

Assume that $\Sigma$ is not compact. 
The topology on $\mathcal E(\Sigma)$ is obtained as follows: (i) any open set 
in $\Sigma$ is open in $\mathcal E(\Sigma)$, and (ii) a basis of open neighbourhoods of an element $\epsilon=\epsilon(\underline M)\in E(\Sigma)$ is the collection $\{V^*(\epsilon, S)\}$ where $S$ is a compact subsurface of $\Sigma$ and $V^*(\epsilon,S):=P_\epsilon\cup V(\epsilon, S)$ with  $P_\epsilon$ being the path component of $\Sigma\setminus S$ that contains $M_n$ for sufficiently large $n$.

Note that if $S\subset int(S')\subset \Sigma$ are compact subsurfaces, then, for any $\epsilon\in E(\Sigma)$, then $V^*(\epsilon, S')\subset V^*(\epsilon, S)$.  

\begin{remark}\label{canonical-neighbourhood} Suppose that $\Sigma$ is a connected non-compact surface. 
A compact subsurface $S\subset \Sigma$ (with boundary) is called {\em canonical} if, for any path component 
$P$ of $\Sigma\setminus S$, the boundary $\partial P$ is homeomorphic to a circle.  A sequence 
$\{S_n\}_{n\in \mathbb N}$ of compact subsurfaces is called a {\em canonical exhaustion} if (i) $\Sigma=\bigcup _{n\ge 1} S_n$, (ii) 
$S_n\subset int (S_{n+1})$, and, (iii) $S_n$ is canonical for any $n$. 
Any connected non-compact surface admits a canonical exhaustion $\{S_n\}.$  See \cite[\S 29A]{ahlfors-sario}. 
If $\{S_n\}$ is a canonical exhaustion of $\Sigma$, then $\{V^*(\epsilon,S_n)\}_{n\ge 1}$ is a basis of neighbourhoods of $\epsilon$ for any 
$\epsilon \in E(\Sigma).$  
\end{remark}

\begin{proposition}\label{surface-union-planar ends}
The topological space $\widetilde \Sigma:=\mathcal E(\Sigma)\setminus E_{\mathrm{np}}(\Sigma)$ is a surface.   
\end{proposition}
\begin{proof}
It suffices to show that if $\epsilon \in E(\Sigma)$ is a planar end, then there exists an open neighbourhood $V\subset \mathcal E(\Sigma)$ of $\epsilon$ such that $V$ is a planar surface.
Choose an open neighbourhood $U$ of $\epsilon$ and an open neighbourhood $U'$ of $E_{\textrm{np}}(\Sigma)$ such that 
$U\cap U'=\emptyset$ and $U\cap E(\Sigma)$ is clopen.  
Such neighbourhoods exist since $\mathcal E(\Sigma)$ is normal, 
$E(\Sigma)$ is homeomorphic to 
a closed subset of the Cantor set, and $E_{\textrm{np}}(\Sigma)$ is closed in $\mathcal E(\Sigma).$  Then\\ (i) $U\setminus E(\Sigma)$ is a surface 
of finite genus, and, \\
(ii) $\partial_{\mathcal E(\Sigma)}U$ is a compact subset of $\Sigma$.\\
Therefore there exists a compact subsurface $K$ of $\Sigma$ such that $V_0:=\Sigma\cap (U\setminus K)$ is a planar surface.  

Let $F$ be a (finite) discrete set of points in bijection with the 
components of $\partial_{\mathcal E(\Sigma)} V_0$ (each of which is a circle).
Then $\mathcal E(V_0)$ is homeomorphic to $ V_0\cup E(V_0)=(U\setminus K)\sqcup F$ and so $U\setminus K=\mathcal E(V_0)\setminus F.$

The Freudenthal compactification of any connected (separable) planar surface, is homeomorphic to $\mathbb S^2$ (\cite[Theorem 1]{richards}).
It follows $\mathcal E(V_0)\cong \mathbb S^2.$  
Therefore $V:=U\setminus K=\mathcal E(V_0)\setminus F $ is a planar surface.  

It follows that $\widetilde \Sigma:=\mathcal E(\Sigma)\setminus E_{\textrm{np}}(\Sigma)$ is a surface.
\end{proof}

\begin{remark}\label{neighbourhoods-of-ends} 
(i)  If $\epsilon \in E_{\mathrm{np}}(\Sigma)$, for any connected open neighbourhood $V\subset\mathcal E (\Sigma)$ of $\epsilon$, the surface $V\setminus E(\Sigma)$ has a non-separating embedded circle and therefore has infinite genus.  
Thus no neighbourhood of $\epsilon$ is homeomorphic to a planar surface.

(ii)  Let $\Sigma=\mathcal E(\Sigma_0)\setminus E_{\mathrm{np}}(\Sigma_0)$ where $\Sigma_0$ is a connected surface.  Then $\Sigma$ is a surface that contains $\Sigma_0$ as a dense open subset of $\Sigma$ and $\mathcal E(\Sigma)=\mathcal E(\Sigma_0)$ and $\Sigma\setminus\Sigma_0=E_{\mathrm{p}}(\Sigma_0)$ is a closed subset of $\Sigma$.   

\end{remark}

\subsection{A classification theorem for perforated surfaces}
The orientability and the genus of a connected surface can be recovered from $\mathring\Sigma$. To see this, 
 let $C\subset \Sigma$ be an embedded circle and let $W\subset \Sigma$ be a closed 
tubular neighbourhood of $C$. 
Then 
$C$ is orientation reversing if and only if any of the following equivalent conditions hold: (1) $W$ is homeomorphic to the closed Möbius band,  (2) $\partial_\Sigma W\cong\mathbb S^1$, and,  (3) $ W\setminus C$ is connected. Equivalently, $C$ is orientation preserving precisely if any one of the following equivalent conditions hold: (1) any (closed) tubular neighbourhood $W$ of $C$ is the trivial $I$-bundle over $C$, (2) 
$\partial W$ is a disjoint union of two copies of $\mathbb S^1$, and, (3) $W\setminus C$ is disconnected.

This allows us to define the following notion: 
\begin{definition}
(i) A simple closed curve $C$ in $\mathring \Sigma$ is {\em orientation preserving (OP)} or 
{\em orientation reversing (OR)} according as whether there is a basis of path connected open neighbourhoods  $\mathcal V=\{V\}$ of $C$ such that for each $V\in \mathcal V, \partial V$ is a disjoint union of two circles or $\partial V\cong \mathbb S^1$ respectively.  \\
(ii) We say that $\mathring \Sigma$ is  {\em non-orientable}  
if there exists an {\em OR} simple closed curve $C$; otherwise every 
simple closed curve in $\mathring \Sigma$ is {\em OP} and 
we say that $\mathring \Sigma$ is {\em orientable.}
\end{definition}

If $\Sigma$ is non-orientable, then there exists an orientation reversing 
simple closed curve in $\mathring \Sigma.$ If $\Sigma$ is infinitely non-orientable, there exists a pairwise disjoint collection $\{C_n\}_{n \ge 1}$ of OR curves in $\mathring\Sigma$ such that $\mathring\Sigma\setminus \bigcup_{n\ge 1}C_n$ is connected and open.  In this case we say that $\mathring\Sigma$ is {\em infinitely non-orientable}.  
  
The following lemma is immediate from the definitions and so we omit the proof.

\begin{lemma}\label{orientability-genus}
    Suppose that $\Sigma$ is any connected surface (without boundary). Then:\\
    (i) $\Sigma$ is orientable if and only if $\mathring \Sigma$ is orientable.\\ 
    (ii)  The genus $g=g(\Sigma)\in \mathbb Z_{\ge 0}\cup\{\infty\}$ equals the maximum number (possibly $\infty$) of pairwise disjoint copies of embedded circles $\{C_j\}$ in $\mathring \Sigma$ such that $\mathring \Sigma\setminus \bigcup_{j} C_j$ is connected.
    \hfill $\Box$
\end{lemma}

Our aim is to show that $\Sigma$ is determined {\em up to its planar ends} by $\mathring \Sigma.$  More precisely, we will show that 
(i) $\mathring\Sigma$ determines the Freudenthal compactification $\mathcal E(\Sigma)$,  
and, 
(ii) if $\mathring\Sigma\cong \mathring \Sigma'$ and $\Sigma,\Sigma'$ do not have planar ends, then $\Sigma\cong \Sigma'.$

\begin{theorem}\label{homeomorphism-extension}
   Suppose $\Sigma, \, \Sigma'$ are two connected, metrizable surfaces. Then any homeomorphism $f: \mathring\Sigma \to \mathring\Sigma'$ extends to a homeomorphism $\mathcal E(f): \mathcal E(\Sigma) \to \mathcal E(\Sigma')$.
\end{theorem}

\begin{proof}
Let $a\in \mathcal E(\Sigma).$ 
    Since $\mathcal E(\Sigma)$ is first countable and $\mathring\Sigma$ is a dense subset, there exists a sequence $(x_n)$ in $\mathring\Sigma$ which converges to $a$ in $\mathcal E(\Sigma)$.  Since $\mathcal E(\Sigma')$ is compact, the sequence $ (f(x_n))$ has a limit point in $\mathcal E(\Sigma')$.   
    
   \emph{ We claim that 
    $(f(x_n))$ converges in $\mathcal E(\Sigma').$} Suppose the contrary, say, $(x_n), (y_n)$ are two sequences in $\mathring\Sigma$, both converging to $a$, such that $(f(x_n)), (f(y_n))$ converge to $b_1, b_2$ in $\mathcal E(\Sigma')$ respectively where $b_1\ne b_2$.  
    Necessarily, $b_i\notin \mathring \Sigma', i=1,2.$  Consider the basis $\mathcal B'=\mathcal B(\Sigma')$ for the topology of $\mathcal E(\Sigma')$ consisting of the following two types of open sets:\\
(i) any open set $B\cong \mathbb R^2$ in $\Sigma'$ with $\partial_{\Sigma'} B$ a circle contained in $\mathring \Sigma'$, \\
(ii) for each $\epsilon'\in E(\Sigma')$, the open sets $V=V^*(\epsilon',S')\subset \mathcal E(\Sigma')$ where $\partial_{\mathcal E(\Sigma')}V$ is a circle 
contained in $\mathring \Sigma'$. (See Remark \ref{canonical-neighbourhood}.)

Choose basic open sets $U_i, V_i\in \mathcal B'$ such that 
$b_i\in V_{i}\subset \overline V_{i}\subset U_i, i=1,2$, and, 
$\overline U_1 \cap \overline U_2=\emptyset$.   Set $C_i':=\partial V_i\cong \mathbb S^1.$  
Here, and in what follows, it is understood that closure and boundaries are taken in 
the Freudenthal compactification. Note that $\mathcal E(\Sigma')\setminus C_i'$
has exactly two path components with common boundary $C_i', i=1,2$. Set $Z_i:=f^{-1}(V_i\cap \mathring \Sigma')$
and $C_i=f^{-1}(C_i')\subset \mathring\Sigma$.  
Then $\mathcal E(\Sigma)\setminus C_i$ has exactly two path components with common boundary $C_i, i=1,2,$ and $Z_i$ is dense in one of them. It follows that  
$\overline Z_i\subset \mathcal E(\Sigma)$ is path connected and that $\partial \overline Z_i$ equals $C_i:=f^{-1}(C_i')$ 
for $i=1,2$.  

We have $a\in \overline Z_i$ for $i=1,2,$ since $x_n\in Z_1$ and $y_n\in Z_2$ for sufficiently large $n$. Also $Z_1\cap Z_2=\emptyset$, which implies that $W:=int(\overline Z_1)\cap int(\overline Z_2)$ is empty. (Otherwise, $\mathring W:=W\cap \mathring \Sigma$ is a non-empty open subset of $Z_1\cap Z_2=\emptyset$, which is absurd.) If $a\in int(\overline Z_1)$, then $y_n\in Z_1$ for sufficiently large $n$, contradicting $Z_1\cap Z_2=\emptyset.$ Therefore $a\in \partial \overline Z_1.$  
Similarly $a\notin int(\overline Z_2)$ and so $a\in \partial \overline Z_2.$ It follows that $a\in \partial \overline Z_1\cap \partial \overline Z_2=C_1\cap C_2=\emptyset$, a contradiction. This proves our claim.

    Thus we have shown that if $(x_n)$ is any sequence in $\mathring \Sigma$ that converges to a point $a$ in $\mathcal E(\Sigma)$, then $f(x_n)$ converges to a 
    (unique) point, say $b\in \mathcal E(\Sigma').$  The desired extension 
    $\mathcal E(f):\mathcal E(\Sigma)\to \mathcal E(\Sigma')$, obtained by sending $a=\lim (x_n)$ to 
    $\lim (f(x_n))=b$, is continuous. For a proof of this, consider an open subset $U'$ of $\mathcal E(\Sigma')$ and let $a \in \mathcal E(f)^{-1}(U')$. Since $\mathcal E(\Sigma'$) is regular, there exists an open subset $V'$ of $\mathcal E(\Sigma')$ for which $\mathcal E(f)(a) \in V' \subset\overline {V'} \subset U'$ and $\partial V'$ is a circle contained in $\mathring\Sigma'$. Let $Z:= f^{-1}(V' \cap\mathring\Sigma')$. Note that $a \in \overline Z$ and that $\partial\overline Z$ is the circle $f^{-1}(\partial V')$. Since $\mathcal E(f)(a) \notin \partial V$, $a \notin \mathcal E(f)^{-1}(\partial V') \supset f^{-1}(\partial V')=\partial \overline Z$. This implies $a \in int(\overline Z) \subset \mathcal E(f)^{-1}(U')$. Thus $\mathcal E(f)^{-1}(U')$ is an open subset of $\mathcal E(\Sigma)$ and hence $\mathcal E(f)$ is continuous.
    
    In order to verify that $f^*$ is a homeomorphism, we need only apply the same argument to $g=f^{-1}$ to obtain a continuous extension $\mathcal E(g): \mathcal E(\Sigma')\to \mathcal E(\Sigma)$ of $g$.  It is readily verified that  $\mathcal E(g)$ is the inverse of $\mathcal E(f)$.
\end{proof}

Now we are ready to prove Theorem \ref{homeo-no-planar-end}.

{\em Proof of Theorem \ref{homeo-no-planar-end}.}
    Note that both $\partial \Sigma_1,\partial \Sigma_2$ are a union of finitely many circles. If $\partial \Sigma_i$ is non-empty, 
    attach a disk along each component of $\partial\Sigma_i$ so as to 
    obtain a surface without boundary $S_i, i=1,2$.    
    Denote by $D_i$ the union of disks so attached 
    and let  
    $ U_i$ the interior of $D_i$ in $S_i.$ 
    Note that $D_1\cong D_2,$ and, 
    $\overline U_i$, the closure of $U_i$ in $S_i$, equals $D_i, i=1,2$.  The homeomorphism 
    $f|_{\partial \Sigma_1}$ extends to a homeomorphism 
    $h:D_1\to D_2.$  
    
 Let $C_1\subset U_1$ be a countable dense subset and set $C_2=h(C_1)$. 
 Set $\mathring S_i:=S_i\setminus (A_i\cup C_i)$.  We obtain a homeomorphism $g:\mathring S_1\to \mathring S_2$ that extends $f$ on $\mathring \Sigma_1$ and  $h|_{U_1\setminus C_1}$ on $U_1\setminus C_1.$ 
 By Theorem \ref{homeomorphism-extension}, $g$ extends to a homeomorphism 
 $\mathcal E(g):\mathcal E(S_1)\to \mathcal E(S_2)$.  Note that $\mathcal E(S_i)=\mathcal E(\Sigma_i)\cup D_i$ and 
 $\mathcal E(g)(\mathcal E(\Sigma_1))=\mathcal E(\Sigma_2)$.  

 Suppose that $\Sigma_2$ has no planar ends.
 If $\epsilon \in E(\Sigma_2),$  then  
 {\em no} neighbourhood of $\epsilon$ is homeomorphic to a planar region.
 (See Remark \ref{neighbourhoods-of-ends}.) 
 It 
 follows that $\mathcal E(g)(x)\notin E(\Sigma_2)$ for any $x\in \Sigma_1$.  Thus $\mathcal E(g)$ restricts to an embedding $\Sigma_1\to \Sigma_2.$  If $\Sigma_1$ also has no 
 planar end, it follows that $\mathcal E(g)$ restricts to a homeomorphism $\Sigma_1\to \Sigma_2$.\hfill $\Box$


\begin{corollary}\label{classification-nested-triples}
   Suppose that $\Sigma_0,\Sigma_1$ are connected surfaces having the same genus and orientation class.  Further, suppose that there exist countable open sets $C_i\subset E_\mathrm{p}(\Sigma_i), i=0,1, $ 
    and a homeomorphism of (nested) triples
   \[\phi:(E_0, E_{\mathrm{np}}(\Sigma_0), E_{\mathrm{no}}(\Sigma_0))\to (E_1,E_{\mathrm{np}}(\Sigma_1),E_{\mathrm{no}}(\Sigma_1))\]
   where $E_i:=E(\Sigma_i)\setminus C_i,~ i=0,1.$  Then $\mathring\Sigma_0\cong \mathring \Sigma_1$.  
   In particular, if $E_\mathrm{p}(\Sigma_i)$ is countable for $i=0,1$, then $\mathring\Sigma_0\cong \mathring\Sigma_1$ if and only if $(E_{\mathrm{np}}(\Sigma_0), E_{\mathrm{no}}(\Sigma_0))
   \cong (E_{\mathrm{np}}(\Sigma_1), E_{\mathrm{no}}(\Sigma_1))$.
   
    Conversely, suppose that $f:\mathring \Sigma_0\to \mathring \Sigma_1$ is a homeomorphism. Then there exists a countable open subset $C_i\subset E_\mathrm{p}(\Sigma_i), i=0,1,$ and a homeomorphism 
    $\widetilde f:\widetilde\Sigma_0\to \widetilde \Sigma_1$ that extends $f$ where 
    $\widetilde \Sigma_i$ is the surface 
    $\mathcal E(\Sigma_i)\setminus (E(\Sigma_i)\setminus C_i),~i=0,1.$

\end{corollary}
\begin{proof}
We shall regard $E(\Sigma_i)$ as subsets of the Cantor set $\mathcal C.$
Since $C_i$ is open in $E_\textrm{p}(\Sigma_i),$ which is further open in $E(\Sigma_i)$,  the set $E_i$ is closed in $E(\Sigma_i)$.
Let $\widetilde \Sigma_i:=\Sigma_i\cup C_i\subset \mathcal E(\Sigma_i).$  
Then $\widetilde \Sigma_i$ is a surface with 
$E_\textrm{p}(\widetilde \Sigma_i)=E_i, E_\textrm{np}(\widetilde \Sigma_i)
=E_\textrm{np} (\Sigma_i), E_\textrm{no}(\widetilde \Sigma_i)=
E_\textrm{no}(\Sigma_i), i=0,1.$  (See 
Remark \ref{neighbourhoods-of-ends}.)
It follows by the Kerékjártó's  classification theorem, $\phi$ yields a homeomorphism $h:\widetilde \Sigma_0\to \widetilde \Sigma_1$.   
Then $\widetilde A_0:=A_0\sqcup C_0$ 
is a countable dense subset of $\widetilde \Sigma_0$ (where $A_i=\Sigma_i\setminus \mathring\Sigma_i$). 
Let $\widetilde A_1=h(\widetilde A_0).$ By the countable homogeneity of 
$\widetilde \Sigma_1,$ we may assume that $\widetilde A_1=A_1\sqcup C_1.$
Then $h$ restricts to a homeomorphism $h_0:\widetilde \Sigma_0\setminus \widetilde A_0\to \widetilde\Sigma_1\setminus \widetilde A_1. $
Since $\widetilde \Sigma_i\setminus \widetilde A_i=\Sigma_i\setminus A_i
=\mathring\Sigma_i,$ the first assertion follows.

We now turn to the converse part.
    Suppose that $f:\mathring \Sigma_0\to \mathring \Sigma_1$ is a 
    homeomorphism.  Let 
    $\mathcal E(f):\mathcal E(\Sigma_0)\to \mathcal E(\Sigma_1)$ be the homeomorphism which extends $f$.  
    Then $\mathcal E(f)$ defines a homeomorphism $f':\Sigma_0'\to \Sigma_1'$
where $\Sigma_i':=\Sigma_i\cup E_\textrm{p}(\Sigma_i)$.   Note that $\Sigma_i'$ is a surface in which $E_\textrm{p}(\Sigma)$ is a closed subset. (See Remark \ref{neighbourhoods-of-ends}.)

Define $C_0:=(\mathcal E(f))^{-1}(\Sigma_1)\cap E_\textrm{p}(\Sigma_0).$ 
Then $C_0$ is an open in $E_\textrm{p}(\Sigma_0).$
Since $\mathcal E(f)$ extends $f$, we have $C_0=(\mathcal E(f))^{-1}(A_1)\cap E_\textrm{p}(\Sigma_0)$ and so $C_0$ is countable. 
Similarly, $C_1:=\mathcal E(f)(\Sigma_0)\cap E_\textrm{p}(\Sigma_1)$
is a countable open subset of $E_\textrm{p}(\Sigma_1).$  Then $\mathcal E(f)
(A_0\cup C_0)= A_1\cup C_1.$ 

Define $\widetilde\Sigma_i:=\Sigma_i\cup C_i=\Sigma_i'\setminus((E_\textrm{p}(\Sigma_i)\setminus C_i).$ Then $\widetilde \Sigma_i$ is a surface with  
$E_\textrm{p}(\widetilde\Sigma_i)=E_\textrm{p}(\Sigma_i) \setminus C_i$.
Now $\mathcal E(f)$ restricts to a homeomorphism $\widetilde f:\widetilde\Sigma_0\to \widetilde\Sigma_1$ which maps the countable 
dense set $\widetilde A_0:=A_0\cup C_0$ onto $\widetilde A_1:=A_1\cup C_1$.
Clearly $\widetilde f$ extends $f:\mathring\Sigma_0\to \mathring\Sigma_1.$
 \end{proof}

We remark that, $\mathring \Sigma\cong\mathring \Sigma_0$ whenever  
$\Sigma:=\mathcal E(\Sigma_0)\setminus (E(\Sigma_0)\setminus C)$ for {\em any} countable open subset $C\subset E_\textrm{p}(\Sigma_0).$ 
If $E_\textrm{p}(\Sigma_0)$ is countable, then we may take $C=E_\textrm{p}(\Sigma_0)$ so that $\Sigma$ is the unique surface without planar ends such 
that $\mathring\Sigma\cong \mathring\Sigma_0.$  If $E_\textrm{p}(\Sigma_0)$
is not countable, we can still find a canonical choice of $\Sigma$ such 
that $\mathring\Sigma\cong\mathring\Sigma_0$ by applying the following lemma, taking $E_0=E_\textrm{p}(\Sigma_0)\subset E(\Sigma_0)=:E.$

\begin{lemma}\label{perfect-open}
    Let $E_0$ be any open subset of a closed subset $E$ of the Cantor set $\mathcal C.$   Then there 
    exists a unique maximal countable open subset $C$ of $E_0$ such that $E_0\setminus C$ is perfect, (possibly empty).   Also $E_0\setminus \ C$ is homeomorphic to one of the spaces 
    $\emptyset, \mathcal C$ or $\mathcal C\setminus \{0\}$. 
\end{lemma}
\begin{proof} 
    If $E_0$ is perfect, then $\overline E_0 \cong \mathcal C$. Since $E_0$ is open in $E$, it is open in $\overline E_0 \subset E$. By Reichbach \cite[Theorem 1]{reichbach}, $E_0 \cong \mathcal C$ or $E_0 \cong \mathcal C \setminus \{0\}$. Therefore, $C=\emptyset$.
    
    Assume that $E_0$ is not perfect. So there exists an isolated 
    point in $E_0.$  Let $\mathcal U$ be a countable basis of $E_0$. Let $\mathcal V=\{U\in\mathcal U\mid U~\mathrm{~is \, countable}\}$.  Then $\mathcal V$ is a nonempty countable collection of open sets in $E_0$ 
    in which each member is countable. 
    Let $C=\bigcup_{V\in \mathcal V} V.$  Then $C$ is open in $E_0$ and is a countable set.  It is readily seen that $E_0\setminus C$ is perfect. 
    Any countable open subset of $E_0$ is contained in $C$ and so $C$ is the unique maximal subset of $E_0$ with the stated property. 

   Set $V:=E_0\setminus C.$  We will first show that $V$ is open in $\overline V.$  
   Since $E_0$ is open in $E$, the set $E_0\cap \overline V$ open in  $\overline V$.   Since $C$ is 
   open in $E_0$, it is also open in $E$.  It follows that $C\cap \overline V=C\cap \overline {E_0\setminus C}=\emptyset.$ So $E_0\cap \overline V\subset E_0\setminus C=V$.   Since $V\subset E_0,$ we conclude that $E_0\cap \overline V=V$ and so $V$ is open in $\overline V.$

Note that since $V$ is perfect, so is $\overline V.$  It follows that 
$\overline V$ is homeomorphic to the Cantor set $\mathcal C.$  
As $V$ is a perfect open subset of $\overline V,$
    by a result of Reichbach \cite[Theorem 1]{reichbach}, we have $V$ is homeomorphic to 
    one of the sets $\emptyset, \mathcal C, \mathcal C\setminus \{0\}$. 
\end{proof}

We now state the classification theorem for connected perforated surfaces.

\begin{theorem}\label{classification-perforated-surfaces}
    (i) Any connected perforated surface is homeomorphic to $\mathring\Sigma$
    where $\Sigma$ is a surface, unique up to homeomorphism, such that $E_\textrm{p}(\Sigma) \in\{\emptyset, \mathcal C, \mathcal C\setminus \{0\}\}$.\\
   (ii)  The collection of all pairwise non-homeomorphic connected perforated surfaces of a fixed genus $g$ and orientation class $\Omega$ is in bijection with the set of homeomorphism types of 
   closed nested triads $(E\supset  E_\mathrm{np}\supset  E_\mathrm{no})$
where $E\subset \mathcal C, E_\mathrm{p}:=E\setminus E_{\mathrm{np}}\cong\emptyset, \mathcal C, \mathrm{~or~} \mathcal C\setminus\{0\}$, 
   and the pairs 
   $(E_\mathrm{np}, E_\mathrm{no})$ and $(g,\Omega)$ are related as in 
   Table 1.\hfill $\Box$
\end{theorem}

\begin{remark}\label{perforated-surfaces-non-homeo}
(i)  It follows from Theorem \ref{classification-perforated-surfaces} that when the genus of $\Sigma$ is finite, $E_\mathrm{p}(\Sigma)=E(\Sigma)$ and 
so $E_\mathrm{p}(\Sigma)\not\cong \mathcal C\setminus \{0\}.$ Thus,
$\mathring \Sigma\cong \mathring S_g$ or $\mathring S_g'$ where $S_g$ is the closed connected surface of genus $g<\infty$, orientable or non-orientable,  and $S'_g=S_g\setminus C$ where 
$C\cong \mathcal C.$  In particular, if $\Sigma $ is a planar surface, then 
$\mathring \Sigma$ is homeomorphic either to $\mathcal N$ or to  $\mathcal N\setminus C$ where $C=\mathcal C\times\{\sqrt 2\}\subset \mathcal N.$

(ii) 
It is a result of Reichbach \cite[Theorem 2]{reichbach} that there are continuously many 
pairwise non-homeomorphic non-empty closed subsets $\{C_\alpha\}_{\alpha\in J}$ of the Cantor set $\mathcal C.$  It follows that there exist surfaces $\Sigma_\alpha, \Sigma'_\alpha$ with $\varepsilon(\Sigma_\alpha)=(C_\alpha, C_\alpha, \emptyset), \varepsilon(\Sigma'_\alpha)=(C_\alpha, C_\alpha, C_\alpha)$. 
By the Kerékjártó classification theorem, $\mathcal F:=\{\Sigma_\alpha, \Sigma'_\alpha\}_{\alpha\in J}$ is a family of pairwise non-homeomorphic surfaces.  Note that none of the surfaces in $\mathcal F$ admit planar ends.
By Corollary \ref{classification-nested-triples}, the corresponding family of perforated surfaces are also pairwise non-homeomorphic.

(iii) Suppose that $\Sigma$ is a surface with genus $g=\infty$. Then it is easy to see that the possibility $E_\mathrm{p}=E_\mathrm{p}(\Sigma)\cong \mathcal C\setminus \{0\}$ does occur.  
Moreover, two such surfaces with  
$(E_\mathrm{np}(\Sigma),E_\mathrm{no}(\Sigma))$ equal to a fixed pair $(E_\mathrm{np},E_\mathrm{no})$ are not, in general, homeomorphic. 
See Example \ref{EpCminuspt} below.
On the other hand, if $ E_\mathrm{p}\cong \mathcal C,$ then there is a unique triad $(E_{\mathrm{p}}\cup E_{\mathrm{np}}, E_{\mathrm{np}},E_{\mathrm{no}})$, up to homeomorphism.  Thus, there is a unique 
perforated surface $\mathring\Sigma$ corresponding to that triad.    
\end{remark}


We recall that the {\em Cantor-Bendixson rank} (or more briefly the {\em rank}) of closed subset of $\mathbb R$ is the smallest ordinal $\lambda$ such that 
$X^{(\lambda)}=X^{(\lambda+1)}$ where, for an ordinal $\alpha$, the subset  $X^{(\alpha)}\subset X$ is defined, using transfinite induction, as follows:
Define 
$X^{(0)}=X$ and $X^{(1)}=X'$, the set of limit points of $X$.  If $\alpha$ is not a limit ordinal, then $X^{(\alpha)}:= (X^{(\alpha-1)})'$.  If $\alpha $ is a limit ordinal, then $X^{(\alpha)}:=\bigcap_{\lambda<\alpha}X^{(\lambda)}.$
The rank of $X$ is a topological invariant of $X$ and so it is independent of its embedding in $\mathbb R$. In fact,
if $f:X\to X$ is a homeomorphism, then $f(X^{(\alpha)})=X^{(\alpha)}$ 
for any ordinal $\alpha$.

\begin{example} \label{EpCminuspt}
    Let $C$ be a countable closed subset of $\mathcal C$ whose 
    Cantor-Bendixson rank equals an infinite ordinal.  
    We have a nested sequence 
    $C\supset C^{(1)}\supset C^{(2)}\supset \cdots $ of non-empty 
    closed subsets of $C$.  Choose a point $\epsilon_n\in C^{(n)}$ for 
    each $n\in \mathbb N$,  
    to be fixed throughout.  
    Let $\Sigma$ be a connected orientable surface with no planar ends such that $E(\Sigma)= E_{\mathrm{np}}(\Sigma)=C.$  For each $n\ge 1$, 
    choose a sequence $D_{n,k}, k\ge 1,$ of disks in $\Sigma$ such that $\lim_{k\to \infty} D_{n,k}=\epsilon_n $  We assume that $\{D_{n,k}\mid k\ge 1,n\ge 1\}$ consists of pairwise disjoint sets.  
    Choose a subset $K_{n,k}\subset D_{n,k}\setminus A\subset \mathring \Sigma$ which is homeomorphic to the Cantor set.    
    We have $\lim_{k\to \infty}K_{n,k}=\epsilon_n.$  Note that $K_n:=\bigcup_{k\ge 1} K_{n,k}$ is a closed subset of $\Sigma$ which is homeomorphic to $\mathcal C\setminus \{0\}$ for each $n$. It follows that the same property holds for $K_J:=\bigcup_{n\in J}K_n$ for any nonempty set $J\subset \mathbb N.$  In particular, $\Sigma\setminus K_{J}$ is a surface for any subset $J\subset\mathbb N$ with $E_{\mathrm{p}}(\Sigma)=\bigcup_{n\in J}K_n$. 
    
    Now 
    let $J\subset\mathbb N$ be any non-empty subset.    
    Define $\Sigma_J:=\mathring \Sigma\setminus K_J$.  We observe that the set of planar ends equals $K_J$ and $E_{\mathrm{np}}=C$.   
    We claim that if $J,J'\subset \mathbb N$ are distinct non-empty subsets, then  $\Sigma_J\not\cong\Sigma_{J'}$.
    To see this, suppose that $n\in J\cup J'\setminus (J\cap J')$---say,  
    $n\in J\setminus J'$.  
    Then $\epsilon_n$ is a limit point of $E_{\mathrm p}(\Sigma_J)$ but $C^{(n)}\cap E_{\mathrm{p}}(\Sigma_{J'})=\emptyset$.  Therefore 
    $(K_J\cup C, C,\emptyset)$ and $(K_{J'}\cup C, C,\emptyset)$ are not homeomorphic triples.  
    Since the planar ends of $\Sigma_{J}$ and $\Sigma_{J'}$ have no 
    isolated points, it follows that $\mathring\Sigma_{J}$ is not homeomorphic to $\mathring \Sigma_{J'}.$  
    Thus, {\em there are continuously many pairwise non-homeomorphic perforated surfaces such with $E_{\mathrm{p}}=\mathcal C\setminus \{0\}$ and $E_{\mathrm{np}}=C$.} 

    \begin{center}
\includegraphics[scale=0.25]{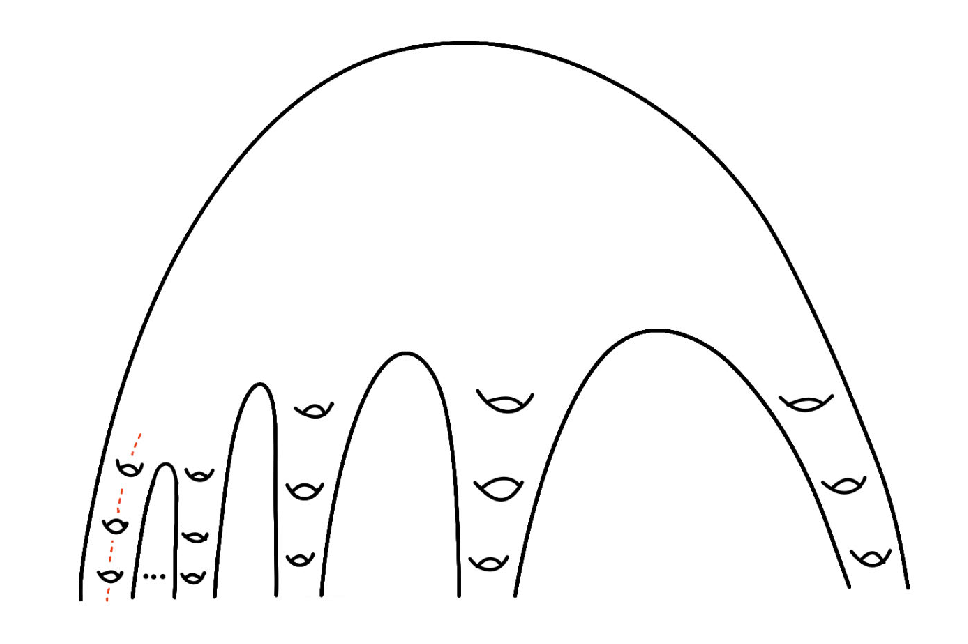}\\
{\em Figure 1: An orientable surface $\Sigma$ with $E_\mathrm{p}\cong \mathcal C\setminus \{0\},
E_\mathrm{np}\cong \{\frac{1}{n}\mid n\in\mathbb N\}\cup \{0\}$, and $0\in \overline{E_\mathrm{p}}$.  Each dashed 
line (in red) represents deletion of a copy of $\mathcal C$.}
\end{center}

\end{example}



   \begin{theorem}  \label{continuously-many-pi1}
   There are $2^{\aleph_0}$ many pairwise non-isomorphic groups 
   of the form $\pi_1(\mathring\Sigma).$
  \end{theorem} 
 \begin{proof} 
By Theorem \ref{classification-perforated-surfaces}(ii) and Remark \ref{perforated-surfaces-non-homeo}(ii) we know that there are continuously 
many pairwise non-homeomorphic connected perforated surfaces. As perforated surfaces are wild,  
by a result of Eda \cite[Theorem 1.3]{eda-2002}, non-homeomorphic perforated surfaces have non-isomorphic fundamental groups. Hence the theorem follows.
 \end{proof}

\subsection{Locally $\mathcal N$-spaces}\label{locally-N}   
 We say that a topological space $X$ is {\em locally $\mathcal N$} if $X$ is Hausdorff, second countable, and has a basis $\mathcal U=\{U_n\}_{n\in \mathbb N}$ such that  
each $U_n$ is homeomorphic to $\mathcal N$ with $\partial_X U_n\cong \mathbb S^1.$  An open set $U\subset X$ such that $U\cong \mathcal N$ with $\partial U\cong \mathbb S^1$ will be referred to as a {\em perforated Jordan region} and its closure $\overline U$, a {\em closed perforated Jordan region}. 

\begin{lemma}
  Any locally $\mathcal N$ topological space is regular.  
\end{lemma}
\begin{proof}
  Let $x \in X$ and $U$ be an open neighbourhood of $x$ which is a perforated Jordan region in $X$.  Then $\overline U=U\cup C$ where $C=\partial U\cong \mathbb S^1$ and $\overline U/C\cong \mathring{\mathbb S^2}$. Consider the quotient map $q: X\longrightarrow X/(X\setminus U)\cong\overline{U}/C.$  It is clear that $q$ maps $U$ homeomorphcally onto its 
  image $q(U)$ in $\overline U/C.$  Since $\overline U/C\cong \mathring{\mathbb S^2}$ is regular, 
  we can find an open neighbourhood $V$ of $q(x)$ such that $V\subset \overline V\subset q(U)$.
  Let $W\subset X$ be the open set defined as $W=q^{-1}(V)\subset U.$  Then $x\in W\subset \overline W\subset U.$ This completes the proof.
\end{proof}

It follows from Urysohn's theorem that $X$ is metrizable, and hence is paracompact.  We will show that $X$ is homeomorphic to a perforated surface.  

Our proof is based on the ideas involved in the proof of T. Rado's classical result,
as given in the book by Ahlfors and Sario \cite[\S8, Chapter-I]{ahlfors-sario}, that any second countable surface has a triangulation. 

\begin{definition}\label{neat-cover} 
Let $X$ be a connected locally $\mathcal N$-space. A countable, locally finite collection 
$\mathcal U=\{U_n\}_{n\in J}$ of closed perforated Jordan regions in $X$ which covers $X$ is called a {\em neat cover} if:\\ (i) for each $n\in J$ there is a homeomorphism $h_n:U_n\to \mathring P_n$ where $P_n\subset \mathbb R^2$ is a closed convex polygon, $\mathring P_n=P_n\setminus A_n$ where $A_n$ is a countable dense subset contained in $int(P_n)$, \\
(ii) if $m,n\in J$ are distinct, then  $U_n\cap U_m=\partial U_m\cap \partial U_n\not\cong\mathbb S^1$ and is connected (possibly empty), \\
(iii) assuming that $U_m\cap U_n\ne \emptyset$, with $m,n\in J$ distinct, $U_m\cap U_n$ is the image of a non-degenerate arc (resp. a singleton), the map $h_{m,n}:h_n(U_n\cap U_m)\to h_m(U_n\cap U_m)$ that sends $h_n(x)$ to $h_m(x)$ is a homeomorphism of an edge (resp. a vertex) of $P_n$ with an edge (resp. a vertex) of $P_m$.
\end{definition}

Suppose that $\mathcal U$ is a neat cover of $X$. By local finiteness of $\mathcal U$, any $x\in X$ belongs to only finitely many members of $\mathcal U$.
Since $X$ is locally $\mathcal N$, if $m,n,k\in J$ are distinct 
if $U_m\cap U_n\cap U_k$ is nonempty, then it is singleton. Given any point $x\in X$, precisely one of the following three possibilities hold:\\ (i) $x\in int(U_m)$ for a unique $U_m\in \mathcal U$, or,\\ (ii)
$x\in U_m\cap U_n$ for exactly two members of $\mathcal U$, 
meeting along a non-degenerate arc, or,\\ (iii) $x\in U_{n_i}\in \mathcal U$ for $0\le i< k$ where $k\ge 3,$ and $U_{n_i}, U_{n_{i+1}}
$ meet along a non-degenerate arc for $0\le i< k$; here it is understood that $U_{n_{k}}:=U_{n_0}$.

\begin{theorem} \label{locally-N-admit-neat-covers} 
    Any locally $\mathcal N$-space $X$ admits a neat cover. \hfill $\Box$
\end{theorem}
We omit the proof of the above theorem.  It is proved following the 
ideas involved in the proof, as given in \cite[\S45-46, Chapter-I]{ahlfors-sario}, of Rado's theorem that any surface having a countable basis has a triangulation. The main observations needed for the proof in the context of a perforated surface (instead of a surface as in \cite{ahlfors-sario}), are the following.  (i) Any perforated surface is locally arcwise connected. (ii) The Jordan Curve Theorem holds for $\mathcal N$, i.e., any simple closed curve $C$ separates $\mathcal N$ into two disjoint open perforated Jordan regions with common boundary $C$.


\begin{theorem}\label{locally-N-are-good}
    Any locally $\mathcal N$-space is homeomorphic to a perforated surface. 
\end{theorem}
\begin{proof}
  
 Choose a neat cover $\{ (U_n,h_n)\}$ of $X$, whose existence is guaranteed by Theorem \ref{locally-N-admit-neat-covers}.  Then $X$ is realized as $\mathring \Sigma$ as follows: Let $ \Sigma=\bigsqcup_{n\in \mathbb N}P_n/\!\sim$ where 
a point $x\in \partial P_n$ is identified with $h_{m,n}(x) \in \partial P_m$
if $h_n^{-1}(x)\in U_n\cap U_m$.  The quotient map $\bigsqcup_{n\in \mathbb N}P_n\to \Sigma$ maps each $P_n$ homeomorphically to its image, again denoted $P_n$.  

Note that, since $X$ is locally $\mathcal N$, for any $n\in J$, each edge of a polygon $P_n$ is adjacent to 
a $P_m$ for a unique $m\in J.$  Moreover, in view of local finiteness of $\mathcal U$, there are only a finite number of polygons $P_n$ which meet at any vertex $v$ of $P_n$. Again, since $X$ is locally $\mathcal N$, there is a cyclic ordering $P_n=P_{n_0}, P_{n_1}, \cdots, P_{n_k}=P_{n_0}$ of the polygons 
containing $v$ such that $P_{n_i},P_{n_{i+1}}$ meet along their common edge. 
It follows that $\Sigma$ is a surface and that 
$X\cong \Sigma\setminus A$ where $A=\bigcup_{n\in \mathbb N} A_n.$
%
%
\end{proof}

\section{Covering spaces}
Since $\mathring\Sigma$ is not semi-locally simply connected (at any point), 
it does not have a simply connected covering. However, it is easy to see that 
it has plenty of path connected coverings.

Let $\mathcal U=\{U_j\}_{j\in J}$ be a covering of $\mathring \Sigma$ consisting 
of path connected open sets. Fix a base point $x_0\in \mathring \Sigma$ and a base point 
$x_j\in U_j$ for each $j\in J$.  
We have the normal subgroup $\pi_1(\mathcal U,x_0)\subset \pi_1(\mathring \Sigma,x_0)$ 
generated by elements of the form $[\sigma\cdot\gamma\cdot\sigma^{-1}] $ where $\gamma$ is a loop in $U_j$ for some $j$ based at $x_j$ and $\sigma$ is a path from $x_0$ 
to $x_j$.  Then, for any subgroup $H\hookrightarrow \pi_1(\mathring \Sigma,x_0)$ that contains $\pi_1(\mathcal U,x_0)$, there exists a covering $p_H:X_H\to \mathring \Sigma$ corresponding to the subgroup $H$, i.e., $p_{H,*}(\pi_1(X_H,\tilde x_0))=H\subset \pi_1(\mathring\Sigma,x_0)$ where $p_H(\tilde x_0)=x_0$. 
Conversely, suppose that $H\subset \pi_1(\mathring\Sigma,x_0)$ corresponds to a covering projection $p:Z\to \mathring \Sigma$, then $H$ contains $\pi_1(\mathcal U,x_0)$ where  
$\mathcal U$ is a covering of $\mathring \Sigma$ consisting of all open sets evenly covered by $p$.    See \cite[Chapter 2,\S 5]{spanier}.  
The deck transformation group of $p_H$, denoted $Deck(p_H),$ is isomorphic to 
the group $N_H/H$ where $N_H$ is the normalizer of $H$ in $\pi_1(\mathring \Sigma,x_0)$.  (By definition, $N_H$ is the largest subgroup of $\pi_1(\mathring \Sigma, x_0)$ in which $H$ is normal.)  When $H$ is normal in $\pi_1(\mathring\Sigma)$, we have $Deck(p_H)\cong \pi_1(\mathring\Sigma)/H.$ (See \cite[Chapter 2, \S6]{spanier}.) 

The following proposition shows that any countable group can be 
realized as the deck transformation group of a covering of $\mathring \Sigma.$ We assume, as we may, that $\pi_1(\Sigma)$ is a free group of countably infinite rank.

\begin{proposition}\label{coverings-countable-deck-group}
Suppose that $\Sigma$ is a connected surface with fundamental group a free group of rank $\aleph_0$.  Given any countable group $H$, 
there exists a regular covering projection $p_0:\mathring \Sigma_0\to \mathring\Sigma$, with $\mathring \Sigma_0$ connected, such that $Deck(p_0)$ is isomorphic to $H$. 
Moreover, $p_0$ is the restriction of a regular covering projection $p:\Sigma_0\to \Sigma.$
\end{proposition}
\begin{proof}
    Suppose that $\gamma_n,n\in \mathbb N,$ are loops based at $x_0\in \Sigma$ such that $\{[\gamma_n]\}_{n\ge 1}$ is a basis for the free group $F:= 
    \pi_1(\Sigma,x_0).$   
    We may (and do) assume that $Im(\gamma_n)\cap A=\emptyset $ for all $n\ge 1$ (where $\mathring\Sigma=\Sigma \setminus A$). We regard $\gamma_n$ as a loop in $\mathring\Sigma$.  
    The inclusion induced homomorphism $i_*:\pi_1(\mathring\Sigma,x_0)\to \pi_1(\Sigma,x_0)=F$ is surjective.   

    Choose a surjective homomorphism $f: F\to H$. 
    This is possible since $|H|\le \aleph_0=rank(F)$. 
    Let $p:\Sigma_0\to \Sigma$ be the regular covering projection corresponding to $\ker (f)$ and let $A_0=p^{-1}(A)$ and $\mathring\Sigma_0:=\Sigma_0\setminus A_0.$
    Then $p_0:\mathring \Sigma_0\to \mathring\Sigma$, defined by the restriction of $p$, is a covering projection with deck transformation 
    group $Deck(p_0)$ the same as $Deck(p)\cong H.$  
\end{proof}
The following corollary is immediate from the above proposition.
\begin{corollary}\label{c-many-coverings} 
     There are $2^{\aleph_0}$ many 
     pairwise inequivalent connected regular covering projections   $p_\alpha:\mathring\Sigma_\alpha\to \mathring\Sigma$. \hfill $\Box$
\end{corollary}


The following proposition, which is a converse to 
Proposition \ref{coverings-countable-deck-group}, is a special case of a more general result due to Conner and Lamoreaux \cite[Theorem 4.1]{conner-lamoreaux}. 
Our proof uses \cite[Theorem 7.3, Lemma 7.6]{cannon-conner}.  We recall the relevant notions and give some details for the sake of completeness. 

Let $\mathcal W$ be an open cover 
of a path connected topological space $X$ and let $\phi: \pi_1(X,x_0)\to G$ be a homomorphism of groups.  Recall (\cite[\S7]{cannon-conner}) that  $\mathcal W$ is said to be 
{\em $2$-set simple relative to  $\phi$}
if, whenever $\gamma:I\to \mathring \Sigma$ is a loop that is freely homotopic to a loop in $X$ whose image is contained in the union of two members of $\mathcal W$, then $[\xi\cdot\gamma\cdot\xi^{-1}]$ belongs to $\ker(\phi)$, where $\xi$ is a path in $\mathring\Sigma$ which connects $x_0$ to $\gamma(0)$. 

Suppose that $X$ is connected and locally path connected separable metric space and $\mathcal W$ is a countable open cover of $X$ whose members are  
path connected.   Suppose that 
$\phi:\pi_1(X,x_0)\to G$ is any homomorphism and $\mathcal W$ is an open cover which is $2$-set simple relative to $\phi$.  Then $Im(\phi)$ is a countable group. See \cite[Theorem 7.3]{cannon-conner}.   

\begin{proposition}\label{coverings-are-countable}
    Suppose that $p:Z\to \mathring \Sigma$ is a path connected covering projection.  
    Then the degree of $p$ is either finite or countably infinite and $Z$ is a perforated surface. 
\end{proposition}
\begin{proof} 
Let $z_0\in Z$ and let $x_0=p(z_0)$ and 
let $H:=p_*(\pi_1(Z,z_0))\subset \pi_1(\mathring\Sigma,x_0)$. 
Let $\mathcal U$ be a covering of $\mathring\Sigma$ such that each $U\in \mathcal U$ is evenly covered by $p$.  
Since $\mathring \Sigma$ is a separable metric space, we may assume that $\mathcal U$ is countable and locally finite. 
Since $\pi_1(\mathcal U,x_0)\subset H$, in order to show that 
the degree of $p$ is countable, we need only 
show that the quotient group $G:=\pi_1(\mathring\Sigma,x_0)/\pi_1(\mathcal U,x_0)$ is countable.  

Denote by $\phi$ the canonical quotient $\pi_1(\mathring \Sigma,x_0)\to G.$
Note that $[\sigma.\gamma.\sigma^{-1}]\in \ker(\phi)$ for any loop $\gamma$
in $U$, for any path $\sigma$ joining $x_0$ to $\gamma(0)$, and any $U\in \mathcal U.$


Put a metric $d$ on $\mathring\Sigma$.
For each $x\in \mathring \Sigma$, let $W_x=B_{r/4}(x)$ be the open ball 
centred at $x$ having radius $r/4$ where $r$ is such that $B_r(x)$ is contained in all the (finitely many)
open sets in $\mathcal U$ that contain $x.$   Then $\mathcal W=\{W_x\mid x\in \mathring \Sigma\}$ is an open cover of $\mathring\Sigma.$  
Suppose that $W_x\cap W_y=B_{r/4}(x)\cap B_{s/4}(y)\ne \emptyset.$  
 Assuming that $ s\le r,$ we see that $W_x\cup W_y\subset B_{r}(x)\subset U$ for any $U\in \mathcal U$ such that $x\in U.$  This proves that $\mathcal W$ is $2$-set simple relative to the canonical quotient map $\phi:\pi_1(\mathring \Sigma,x_0)\to G$.
Let $\mathcal V\subset \mathcal W$ be a countable subcover of $\mathcal W$.
Then $\mathcal V$ is also $2$-set simple relative to $\phi$.

It follows, by \cite[Theorem 7.3]{cannon-conner}, 
that $Im(\phi)\subset G$ is countable. Since $\phi$ is surjective, we conclude that $
G=\pi_1(\mathring \Sigma,x_0)/\pi_1(\mathcal U,x_0)$ is a countable group.

Since the degree of $p$ is countable, $Z$ is second countable. Since $p$ is a local homeomorphism, and since $\mathring \Sigma$ is has a covering by evenly covered open sets 
which are perforated Jordan regions, it follows that $Z$ admits a countable open covering by 
perforated Jordan regions.  Since $\mathring \Sigma$ is Hausdorff, so is $Z$.  Thus $Z$ is locally $\mathcal N$. By Theorem \ref{locally-N-are-good}, $Z$ is a perforated surface.
\end{proof}


Suppose that $N=\ker(\phi)$ where $\phi:\pi_1(\mathring\Sigma,x_0)\to L$ is a homomorphism into a countable abelian group with no infinitely divisible elements. Then there exists a covering $p_0:\mathring \Sigma_0\to\mathring \Sigma $ 
such that $p_{0,*}(\pi_1(\mathring \Sigma_0,\tilde x_0))$ equals $N$ (where $p_0(\tilde x_0)=x_0$).  Indeed, by \cite[Theorem 4.4]{cannon-conner}, there 
exists an open neighbourhood $U\subset \mathring \Sigma$ around $x_0$ such 
that $\phi(\pi_1(U,x_0))$ is trivial, i.e., $\pi_1(U,x_0)\subset N.$ 

Now let $x\in \mathring\Sigma$ be arbitrary. Choose a path $\alpha$ in $\mathring\Sigma$ joining $x_0$ to $x$, and consider the homomorphism $\phi_x$ defined as the composition 
$\pi_1(\mathring\Sigma, x)\stackrel{\alpha_*}{\to}\pi_1(\mathring\Sigma,x_0)\stackrel{\phi}{\to} L$.  We obtain an open  neighbourhood 
$U_x$ of $x$ such that $\pi_1(U_x,x)\subset \ker(\phi_x)$. Therefore, $\alpha_*(\pi_1(U_x,x))\subset N$. 
Taking $\mathcal U=\{U_x\}_{x\in X}$, using the normality of $N$, we have $\pi_1(\mathcal U)\subset N$.   
This implies the existence of 
the covering projection $p_0:\mathring \Sigma_0\to \mathring \Sigma$ corresponding to the subgroup $N$.    

Let $p$ be a prime. 
We give below, an example of an index $p$ normal subgroup $N$ of $\pi_1(\mathring \Sigma)$ which does not correspond to any covering of $\mathring \Sigma$.

\begin{example}\label{index-p-but-not-correspond-to-cover}
Let $p\in \mathbb N$ be a prime. For $k\ge 1$, let $\xi_k:\mathbb S^1\to \mathcal H=\bigvee_{k\ge 1}C_k$ be the loop based at $0$ which maps onto the circle $C_k$ homeomorphically.  
Then the homology class of $\xi_k$ generates an infinite cyclic subgroup, which is a summand of 
$H_1(\mathcal H;\mathbb Z).$  In particular, it represents a non-zero element $x_k$ in the 
$\mathbb Z_p$-vector space $H_1(\mathcal H;\mathbb Z_p)$.  It is not difficult to show that $\mathcal X:=\{x_k\mid  k\in \mathbb N\}$ is linearly independent. Let $\mathcal B$ be a basis 
of $H_1(\mathcal H;\mathbb Z_p)$ that contains $\mathcal X.$  Let $\eta:H_1(\mathcal H;\mathbb Z_p)\to \mathbb Z_p$ be the linear map defined as $\eta(x_k)=1\in \mathbb Z_p~\forall k\in \mathbb N,$ 
and $\eta (b)=0~\forall b\in \mathcal B\setminus \mathcal X.$  
Let $j:(\mathcal H,0)\hookrightarrow (\mathring \Sigma,x_0)$ be an embedding and let $r:\mathring\Sigma \to \mathcal H$ be a retraction. 
We have a surjective homomorphism $\phi:\pi_1(\mathring \Sigma)\to \mathbb Z_p$ defined as the composition 
\[
\pi_1(\mathring \Sigma)\stackrel{r_*}{\to} \pi_1(\mathcal H)\to H_1(\mathcal H;\mathbb Z_p)\stackrel{\eta}{\to}\mathbb Z_p.
\]
Let $N=\ker(\phi)$ which is of index $p$ in $\pi_1(\mathring \Sigma).$  Then $N$ cannot correspond to any covering of $\mathring \Sigma.$  Indeed, if $U$ is any open set containing $x_0$, there is a $k$ such that $Im(\xi_k)\subset U$ and we have $\phi(x_k)\ne 0$.  So $x_k\notin N.$
\end{example}

It is easily seen that there are plenty of (pairwise non-isomorphic) normal subgroups of $\pi_1(\mathring\Sigma)$.  However, the following theorem implies that there are no nontrivial countable normal subgroups in $\pi_1(\mathring\Sigma)$.

\begin{theorem} Suppose that $X$ is a connected, locally path connected, $1$-dimensional separable metric space which has a wild point $x_0\in X$.  Then
the conjugacy class of any nontrivial element of $\pi_1(X,x_0)$ has cardinality $\mathfrak c$. In particular,  
if $N$ is a nontrivial normal subgroup of $\pi_1(X,x_0),$ then $|N|=\mathfrak c.$   
\end{theorem}
\begin{proof}  Let $[\gamma]\in N$ be a nontrivial element.  We assume that $\gamma$ is a reduced loop.

Let $U$ be an open neighbourhood of $x_0$ such that $Im(\gamma)$ is not contained in $U$.  Since $x_0$ is a wild point, $U$ does not admit a universal cover. Therefore, by \cite[Corollary 5.3, Lemma 5.4]{cannon-conner}, $\pi_1(U,x_0)$ is uncountable.  In fact, using infinite `legal' product of loops, it can be seen that 
$\pi_1(U,x_0)$ has cardinality $\mathfrak c.$   We need to show that 
if $[\eta],[\zeta]\in \pi_1(U,x_0)$ are distinct elements, then $[\eta][\gamma][\eta]^{-1}\ne [\zeta][\gamma][\zeta]^{-1}$.  Equivalently, we shall show that if $[\xi]\in \pi_1(U,x_0)$ is nontrivial, then $[\xi][\gamma]\ne [\gamma][\xi].$ 

Let $\xi:I\to U$ be any reduced loop which is not null homotopic.   Then   
$[\xi][\gamma]\ne [\gamma][\xi]$.  For, otherwise, the subgroup $H$ generated by $[\xi],[\gamma]$ is not a free group of rank $2$ and hence must be cyclic, in view of the fact that $\pi_1(X;x_0)$ is locally free. Suppose that $[\zeta]$ is a generator of $H$ where $\zeta$ is chosen to be reduced.  Write  
$[\gamma]=[\zeta]^m, [\xi]=[\zeta]^n$.  Then $m,n$ are non-zero. Since $\gamma,\xi, \zeta$ are all reduced, their images must all be the same. This is a contradiction 
to our choice of $U$ and $\xi$. 
%
This completes the proof.  
\end{proof}

Let $p:Z \to \mathring\Sigma_0$ be a covering projection with $Z$ path connected. Then by Proposition \ref{coverings-are-countable}, $Z \cong \mathring\Sigma_1$ for a Hausdorff, second countable surface $\Sigma_1$. We have the following result.

\begin{theorem}
    Let $p:\mathring \Sigma_1 \to \mathring\Sigma_0$ be a covering projection where $\Sigma_1$ is connected. Then there exists  
    a covering projection $p': \Sigma_1' \to \Sigma_0'$ of surfaces,  
    and embeddings $i_k:\mathring\Sigma_k\longrightarrow  \Sigma_k'$, $k=0,1,$  such that (i) $\Sigma_k'\setminus i_k(\mathring\Sigma_k)\subset \Sigma_k'$ is a countable dense subset and, (ii) 
the following diagram commutes:
\[
\begin{tikzcd}
   \mathring\Sigma_1 \arrow[r,"i_1"] \arrow[d,"p"] & \Sigma_1' \arrow[d,"p'"]\\
   \mathring\Sigma_0 \arrow[r,"i_0"] & \Sigma_0'.
\end{tikzcd}
\]
\end{theorem}

\begin{proof} 
 An open subset of $\mathring\Sigma_k$, $k=0,1$, is of the form $W \cap \mathring\Sigma_k$, where $W$ is an open subset of $\Sigma_k$. For the rest of the proof, we will denote sets of the form $W \cap \mathring \Sigma_k$ by $\mathring W$. There exist countable open covers $\mathcal U, \, \mathcal V$ of $\mathring\Sigma_0$ which satisfy the following properties:\\
  (i) Any $ \mathring W\in \mathcal U\cup \mathcal V$ equals $W\cap \mathring\Sigma_0$ where $W$ is homeomorphic to the unit ball in $\mathbb R^2$ and 
  $\partial \mathring W\subset \mathring \Sigma_0$ is a circle. \\ 
  (ii) For each $\mathring V \in \mathcal V$, there exists $\mathring U \in \mathcal U$ such that $\overline V=\textrm{cl}_{\Sigma_0}V \subset U$ and the circle $\partial V=\partial_{\mathring\Sigma_0}\mathring V$ is contained in $\mathring U$.\\ 
  (iii) All members of $\mathcal U, \, \mathcal V$ are evenly covered by $p$.

Write $\mathcal V=\{\mathring V_n\}$. For each $n$, there exists $\mathring U_n \in \mathcal U$ such that $\overline V_n \subset U_n$. Since $\mathring U_n$ is evenly covered by $p$, $p^{-1}(\mathring U_n)= \bigsqcup_k \mathring U_{n,k}$, where the $U_{n,k}$ are open subsets of $\Sigma_1$ and the restrictions $ \mathring U_{n,k} \to \mathring U_n$, denoted by $p_{n,k}$, are homeomorphisms. Denote the circle $\partial V_n$ by $C_n$, $p_{n,k}^{-1}(\partial V_n)$ by $C_{n,k}$ and the set $p_{n,k}^{-1}(\mathring V_n)$ by $\mathring V_{n,k},$ where $\mathring V_{n,k}= V_{n,k} \cap \mathring\Sigma_1$, $V_{n,k}$ is an open subset of $\Sigma_1$. We can (and do) choose $V_{n,k}$ to be such that $\partial_{\Sigma_1}(V_{n,k})=C_{n,k}$. Note that $\bigcup V_n$ is a dense open subset of $\Sigma_0$. In general, $\{V_n\}$ need not cover $\Sigma_0$ and $\Sigma_0 \setminus \bigcup V_n$ is a countable closed subset. Similarly, $\Sigma_1 \setminus \bigcup_{n,k}V_{n,k}$ is a countable closed subset of $\Sigma_1$.

The set $\overline V_{n,k}=V_{n,k} \cup \partial V_{n,k}$ is a subsurface of $\Sigma_1$ with one boundary component. The map $p_{n,k}: \overline V_{n,k} \cap \mathring\Sigma_1 \to \overline V_n \cap \mathring\Sigma_0$ is a homeomorphism of perforated surfaces with boundary. By Theorem \ref{homeo-no-planar-end}, it extends to a homeomorphism $\mathcal E(p_{n,k}): \mathcal E(\overline V_{n,k}) \to \mathcal E(\overline V_n)$. Since $\overline V_n$ is a closed disk, $E(\overline V_n)=\emptyset$ and $\mathcal E(\overline V_n)= \overline V_n$. Let $V_{n,k}'=V_{n,k} \cup \, E(\overline V_{n,k}) \subset \mathcal E(\overline V_{n,k})$. Since $\mathcal E(p_{n,k})$ maps $\partial V_{n,k}$ homeomorphically onto $\partial V_n$, $\mathcal E(p_{n,k})$ restricts to a homeomorphism $\mathcal E(p_{n,k}): V_{n,k}' \to V_n$.
  
By Corollary \ref{classification-nested-triples}, $E(\overline V_{n,k})$ is equal to $E_{p}(\overline V_{n,k})$ and is a countable set. Note that the set $E(\overline V_{n,k})$ embeds canonically in $E_p(\Sigma_1)$ as an open subset. 
We identify $E(\overline V_{n,k})$ with its image in $E_p(\Sigma_1)$. Set $E:=\bigcup_{n,k} E(\overline V_{n,k})$. Then $E$ is a countable open subset of $E_p(\Sigma_1)$.
In view of Proposition \ref{surface-union-planar ends}, $\Sigma_1'' := \Sigma_1 \cup E \subset \mathcal E(\Sigma_1) \setminus E_{np}(\Sigma_1)$ is a surface. Since $V_{n,k}'$ is an open subset of $\Sigma_1''$, so is $\Sigma_1'= \bigcup_{n,k} V_{n,k}'$. Hence $\Sigma'_1$ is a surface. Define $\Sigma_0':=\bigcup_n V_n$. Then $\Sigma_0'$, being an open subset of $\Sigma_0,$ is also a surface. 
We have natural inclusions $i_k: \mathring\Sigma_k \to \Sigma_k'$, for $k=0,1$.

Define $p': \Sigma_1' \to \Sigma_0'$ by $p'|_{V_{n,k}'}=\mathcal E(p_{n,k})$.  
   Note that for any two pairs $(n,k)$ and $(m,l)$, if $V_{n,k}' \cap V_{m,l}' \neq \emptyset$, the maps $\mathcal E(p_{n,k})$ and $\mathcal E(p_{m,l})$ agree on the intersection. So we get a well-defined surjective continuous map $p': \Sigma_1' \to \Sigma_0'$. For any $n \in \mathbb N$, $(p')^{-1}(V_n)= \bigsqcup_{k} V_{n,k}'$. Further, for $n$ fixed, $\{V_{n,k}'\}_k$ is a disjoint collection of open subsets of $\Sigma_1'$. This is because $\mathring V_{n,k} \cap \mathring V_{n,l} = \emptyset$ for $k \neq l$, which implies $V_{n,k} \cap V_{n,l}= \emptyset$. Also, the boundary circle of $V_{n,k}$ in $\Sigma_1$ separates $E(\overline V_{n,k})$ from $E(\overline V_{n,l})$ for $k \neq l$.
   
   This proves that $p'$ is a covering map. Since $p'$ extends $p$, the diagram in the statement of the theorem commutes.
\end{proof}

\subsection{Normal subgroups of $\pi_1(\mathring\Sigma)$} \label{normal-subgroups}
In view of Corollary \ref{c-many-coverings} there are continuously many pairwise distinct subgroups of $\pi_1(\mathring \Sigma)$ each having countable index. There are also a large collection of normal subgroups each having index $\mathfrak c$. For example, if $x\in \mathring\Sigma$, $x\ne x_0$, then it can be shown that the normal subgroup $N_x=\langle\langle \pi_1(\mathring\Sigma\setminus \{x\})\rangle\rangle$ of $\pi_1(\mathring\Sigma,x_0)$, generated by $ \pi_1(\mathring\Sigma\setminus\{x\},x_0)$, has index $\mathfrak c$.  The same is true of the normal subgroup $N_C=\langle\langle\pi_1(\mathring\Sigma\setminus C,x_0)\rangle\rangle$ where $C\subset \mathring \Sigma$ is any closed subset such that $x_0\notin C$.  Evidently $N_C\subset N_{C'}$ if $C'\subset C$. The inclusion $N_C\subset N_{C'}$ is strict if $C'$ is a proper subset of $C$. 
If $C=\bigcap_{k\ge 1} C_k$ where 
$C_{k+1}\subset C_{k}$ for all $k\ge 1$, then $N_C=\bigcup_{k\ge 1} N_{C_k}$.

There are other constructions of normal subgroups 
of $\pi_1(\mathring\Sigma)$ which do not correspond to 
any covering of $\mathring\Sigma.$  For example, consider 
the normal subgroup $N$ of $\pi_1(\mathring\Sigma,x_0)$ generated by elements of the form $[\xi\cdot\sigma\cdot \xi^{-1}]$ as $\sigma$ varies over simple closed curves in $\mathring \Sigma$ based at a point $x$ and $\xi$ is any path joining $x_0$ to $x.$  
It is easily seen that $N$ does not contain $\pi_1(\mathcal U)$ for any open cover $\mathcal U$ of $\mathring\Sigma.$  Again, $N$ has 
index $\mathfrak c$ in $\pi_1(\mathring\Sigma,x_0).$   

\section{Embeddings and retractions}
We shall obtain some embedding results showing 
that the fundamental groups of perforated surfaces are 
large.  By a result of Cannon and Conner \cite[Corollary 3.3]{cannon-conner}, any inclusion $j:X\hookrightarrow Y$
between path connected $1$-dimensional Hausdorff spaces 
induces a monomorphism.  We shall identify $\pi_1(X)$ with its image under $j_*$.  When $X$ is a retract of $Y$, $\pi_1(X)$ is a retract of  $\pi_1(Y).$

Let $D\hookrightarrow \Sigma$ be an embedded disk such that, setting $C:=\partial D$, we have $C\cap A=\emptyset$. Here $A\subset \Sigma$ is a countable dense subset and $\mathring\Sigma=\Sigma\setminus A.$
Let $\mathring D=D\cap \mathring\Sigma.$  Then, writing 
$V=\mathring \Sigma\setminus int(D)$, we have 
$\mathring D/C\cong \mathring \Sigma/V$. Thus we have a 
quotient map $q:\mathring \Sigma\to \mathring D/C$.  Suppose that $X$ is a compact subspace of $int(\mathring D)$.  Then $q|_X$ is an embedding. 
Suppose that $r:\mathring D/C\to X$ is a retraction. 
Then the composition \[\mathring\Sigma\stackrel{q}{\to}\mathring D/C\stackrel{r}{\to} X\] 
is a retraction.  

Note that $\mathring D/C\cong \mathring{\mathbb S}^2$, arising from the identification $D/C\cong \mathbb S^2$, in which the point $\{C\}$ corresponds to $\infty$.  
  

In Lemma \ref{criterion-for-retraction} below, we obtain a criterion for a compact subspace $X\subset \mathcal N$ to be a retract.

 \begin{definition} \label{property-R} (a) Suppose that $X\subset \mathcal N$ is a compact 
 path connected subspace 
 such that for any connected component $U$ of $\mathcal N\setminus X$, there is a relative homeomorphism 
 $h: (\mathring{D},\partial\mathring D)\to (\overline {U},\partial U)$ where $\mathring D$ is a perforated disk with $\partial \mathring D\cong \mathbb S^1$.
 We say that $X$ has {\em Property $R$} if any one the following holds:\\ (i) $\mathcal N\setminus X$ has finitely many path components.\\ 
(ii) Suppose that $\{U_n\}_{n\in \mathbb N}$ are the path components of $\mathbb S^2\setminus X$.   If there exists a convergent sequence $(y_k)$ in $\mathcal N$ with $\lim_{k\to \infty} y_k=y_0\in X$ where $y_k\in U_{n_k}$, $U_{n_k}\cap U_{n_l}=\emptyset$ for 
$k\ne l $, then $\lim_{k\to\infty} U_{n_k}=\{y_0\}$, that is, any sequence $(u_{n_k})$ with $u_{n_k}\in U_{n_k}$ converges to $y_0$. 

(b) We say that a compact subspace $X\subset \mathring \Sigma$ has property $R$ if $X\subset B$ where $B\subset \mathring \Sigma$ is open, $\partial B\cong \mathbb S^1$, and there exists an embedding $h:\overline B\to \mathcal N$ such that $\partial h(B)\cong \mathbb S^1$, and, 
$h(X)\subset \mathcal N$ has property $R$.  
 \end{definition} 


 



\begin{lemma} \label{criterion-for-retraction}
  Suppose that $X\subset \mathcal N$ has Property $R$.  Then $X$ is a retract of $\mathcal N.$
\end{lemma}
\begin{proof} Suppose that $\{U_n\}$ are the path components of $\mathcal N\setminus X$.  
Let $h_n:(\mathring{ D}, \partial D)\to (\overline{U}_n, \partial U_n)$, be a relative homeomorphism. We shall assume that $0\notin \mathring D.$
Then the radial retraction $r:D \setminus\{0\}\to \mathbb S^1$ defines a retraction  
$\rho_n:\overline U_{n}\to \partial U_n$.  Explicitly, \\
\[ \rho_n(y)= \begin{cases}
    h_n (r(h_n^{-1}(y))), &~\mathrm{if~} y\in U_n,\\
    y,& ~\mathrm{if}~y\in \partial U_n.
\end{cases}
\]
Now define $\rho:\mathcal N\to X$ as $\rho(x)=\rho_n(x)$ if $x\in U_n $ and $\rho(x)=x$ if $x\in X$.
Since $X$ has property $R$, it follows that the map 
$\rho:\mathcal N\to X$ is {\em continuous}.  
Thus $X$ is a retract of $\mathcal N.$ 
\end{proof}
\begin{corollary}\label{retraction-fundamental-group}
    Suppose that $X\subset \mathring \Sigma$ has property $R$. Then $X$ is a retract of $\mathring \Sigma$. 
    Consequently, there is a retraction $\pi_1(\mathring\Sigma)\to\pi_1(X).$ 
\end{corollary}
\begin{proof}  Choose an open set $B\subset \mathring \Sigma$ and an embedding $h:\overline B\to \mathcal N$ such that $X\subset B\subset \mathring\Sigma$, $\partial B\cong \mathbb S^1$, and $h(X)\hookrightarrow\mathcal N$ has property $R$. 
Set $V:=\mathring\Sigma\setminus B.$   We have a quotient map 
$q:\mathring \Sigma\to \mathring\Sigma/V$ which defines a 
homeomorphism $q_0:\overline B/\partial B\to \mathring \Sigma/V.$

The map $h$ defines a {\em homeomorphism} 
$\overline B/\partial B\stackrel{\bar h}{\longrightarrow}  \mathcal N/V'$ where $V'=\mathcal N\setminus h(B)$. 
There is also homeomorphism $\lambda :\mathcal N/V'\to \mathcal N$ which is identity in a neighbourhood of $h(X)$.  Thus, $X':=\lambda (\bar h(X))=h(X)$. 
By Lemma \ref{criterion-for-retraction}, there is retraction $\rho:\mathcal N\to X'.$    This yields a retraction 
$r: \overline B/\partial B\to X.$  Now the composition 
\[\mathring \Sigma\to \mathring \Sigma/V\stackrel{q_0^{-1}}{\longrightarrow} \overline B/\partial B\to X\] is a retraction. 


\end{proof}


 \begin{example}\label{fundamentalgroups-as-retracts}
(i)  The Hawaiian earring $\mathcal H$ has the property R. So do 
the Sierpiński curve $\mathcal S$ and the Sierpiński gasket $\mathcal T$. 
Hence $\mathcal H,\mathcal S,\mathcal T$ are 
all retracts of $\mathcal N$ and their fundamental groups  
embed in $\pi_1(\mathring \Sigma)$ as retracts for any surface $\Sigma$.\\
(ii) Let $S=\mathcal S\cup J\subset \mathbb R^2$ where $J$ is the straight line segment $[-1,0]\times \{0\}$, with base point $x_0:=(-1,0)\in S$. Then $S$ is locally simply connected at $x_0$. 
Consider the ``weak join" 
$\mathbf S=\bigvee_{n\in \mathbb N} S_n\subset \prod_{n\in \mathbb N}S_n$, with subspace topology, where each $(S_n,x_n)$ is a copy of $(S,x_0)$ with base point $y=(x_n)_{n\in \mathbb N}$.  Then, by a result of Morgan and Morrison \cite[Theorem 4.1]{morgan-morrison}, $\pi_1(\mathbf S)\cong\dtimes_{n\in \mathbb N}\pi_1(S_n,x)\cong \dtimes_{\mathbb N}\pi_1(\mathcal S).$   It is easy to see that $\mathbf S$ can be embedded in $\mathcal N$ and hence $\pi_1(\mathbf S)\cong \dtimes_\mathbb N\pi_1(\mathcal S)$ embeds in $\pi_1(\mathcal N).$ 

\begin{center}
    \includegraphics[scale=0.25]{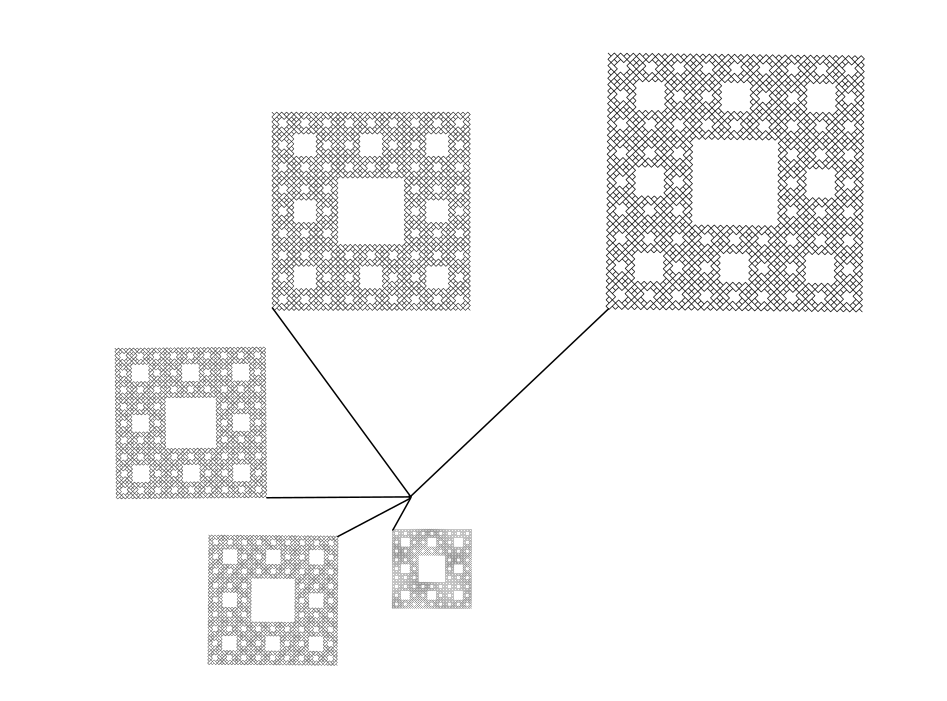}\\
    {\em Figure 2: The space $\mathbf S=\bigvee_{n\in \mathbb N} S_n$.\\
    The image of the Sierpiński curve is courtesy Wikipedia:\\  https://en.wikipedia.org/wiki/Sierpiński\_carpet }
\end{center}

(iii) Let $Y=I\times \{0,1\}\cup\{\frac{1}{n}\mid n\in \mathbb N\}\times I\cup \{0\}\times I.$  Then $Y$ does not have the property $R$. 
\end{example}

\begin{theorem}\label{freeproduct-of-pi1-sigma-perforated}
    Suppose that $\Sigma$ is a connected surface.
    Then the following groups embed in $\pi_1(\mathring\Sigma)$:\\ 
    (i) 
     $(\dtimes_{\mathbb N}\pi_1(\mathcal N)) *\pi_1(\mathring\Sigma)$, and, \\
(ii) the free product $*_\mathfrak c\pi_1(\mathring \Sigma)$ of $\mathfrak c$ many copies of $\pi_1(\mathring \Sigma)$.
\end{theorem}
\begin{proof} 
(i) Example \ref{fundamentalgroups-as-retracts}(ii) shows that $\pi_1(\mathbf S)\cong \dtimes_{\mathbb N}\pi_1(\mathcal S)$ embeds in $\pi_1(\mathcal N)$. Since $\mathcal N$ embeds 
in $\mathring\Sigma,$ we have embeddings $\pi_1(\mathbf S)\to \pi_1(\mathcal N)\to \pi_1(\mathring\Sigma)$.  On the other hand, $\pi_1(\mathcal N)$ embeds in $\pi_1(\mathcal S)$
by \cite[Corollary 3.4]{cannon-conner}. (In fact $\mathcal N$ itself embeds in the Sierpiński curve $\mathcal S$.) Therefore we have the following sequence of embeddings:
\begin{equation}\label{embeddings-into-pi1N}
\dtimes_\mathbb N\pi_1(\mathcal N)\to \pi_1(\mathbf S) \to \pi_1(\mathcal N). 
\end{equation}



Let $X=U\cup J\cup V\subset B_3(0)\cap \mathcal N$ where $U=\{v\in \mathcal N\mid ||v||<1\}, V:= \{v\in \mathcal N\mid 2<||v||<3\}$ and  
$J\subset \mathcal N$ is a straight line segment joining two points $u, v\in \mathcal N$ such 
that $||u||=1,||v||=2.$  We regard $V$ as an annulus $\mathring B_3\setminus D_2$ in $ \mathring\Sigma\setminus D$ where $D_2\subset B_3\subset \Sigma$, $B_3\cong B_3(0)$ an open ball and $D_2=\overline B_2(0)$ is a closed disk.  We regard $U$ as a subspace of $D_2\cap\mathring \Sigma$ and $J\subset B_3$.  Now $X\subset U\cup J\cup (\mathring\Sigma\setminus D_2)\subset \mathring\Sigma.$  Therefore we have a monomorphism $\pi_1(U\cup J\cup (\mathring\Sigma\setminus D_2))\to\pi_1( \mathring\Sigma).$  Since $\mathring\Sigma\setminus D_2
\cong \mathring\Sigma$ and since $\pi_1(U\cup J\cup(\mathring\Sigma\setminus D_2))$ contains 
$\pi_1(U)*\pi_1(\mathring\Sigma\setminus D_2)$, we obtain an embedding of $\pi_1(\mathcal N)*\pi_1(\mathring\Sigma)\cong \pi_1(U)*\pi_1(\mathring\Sigma\setminus D_2)$ in $\pi_1(\mathring\Sigma)$.  Taking free product with $\pi_1(\mathring\Sigma)$ on both sides of Equation \ref{embeddings-into-pi1N}, we obtain 
an embedding $(\dtimes_\mathbb N\pi_1(\mathcal N))*\pi_1(\mathring\Sigma)\to \pi_1(\mathring\Sigma).$

(ii) It is easy to see that $F_\mathfrak c$ embeds in $\pi_1(\mathcal N)$. Hence $F_\mathfrak c*\pi_1(\mathring\Sigma)$ embeds in $\pi_1(\mathring\Sigma)$. Evidently, the former group 
contains a copy of $*_\mathfrak c\pi_1(\mathring\Sigma)$.
\end{proof}

 \section{Non-Hopficity of the fundamental group.}
Recall that  a group $G$ is said to be Hopfian if every surjective homomorphism $G\to G$ is an isomorphism.  It is well-known that free groups of finite rank are Hopfian.  It is trivial to see that free groups of infinite rank are not Hopfian.  So is the fact that  
$\pi_1(\mathcal H)$ is not Hopfian.  Also, there exist one-dimensional 
non-compact spaces whose fundamental groups are not Hopfian.  For example, the space $S=\bigcup_{n\ge 1}S_{n}\subset \mathbb R^2$, where $S_n$ is the circle 
with centre $(n/2,0)$ and radius $n/2$, has fundamental group isomorphic to the free group of rank $\aleph_0$, which is non-Hopfian.  

Our aim here is to show that $\pi_1(\mathring \Sigma)$ is not Hopfian for any 
surface $\Sigma$. 
We shall also show that $\pi_1(X)$ is non-Hopfian when $X$ is the Sierpiński curve $\mathcal S$, the Menger curve
$\mathcal M$, or the Sierpiński gasket $\mathcal T$. 

Before proceeding with the proof, we observe the following. 
\begin{remark}
Consider the quotient 
map $q:\mathring\Sigma\to \mathring\Sigma/\mathring D_0$, where  
$ D_0\subset int(D) \subset \Sigma$, $D$ being an embedded disk with $\partial D\cap A=\emptyset=C_0\cap A$ where $C_0:=\partial D_0$.   
We have $\mathring\Sigma/\mathring D_0\cong \mathring\Sigma.$  However, the induced 
homomorphism $q_*:\pi_1(\mathring\Sigma)\to \pi_1(\mathring\Sigma)$ is {\em not} surjective. To see this, let $x_0\in \mathring D_0$ to be the base of $\mathring\Sigma$ and let $y_0:=q(x_0)=\{\mathring D_0\}$.  

To proceed further, 
it will be convenient 
to identify $D$ with $\overline{B}_{\pi}(0)\subset \mathbb R^2$, and $D_0$ with $\overline B_{\frac{\pi}{2}}(0)$ so that $C_0$ is the circle with centre $0$ and radius $\frac{\pi}{2}$.  Let $C_n$ be the circle with centre $0$ and radius $r_n:=\frac{\pi(n+1)}{2n}$ so that $C_n\subset \mathring D\setminus D_0$. 
Let $\sigma:I\to \mathring D$ be a straight line segment joining  
$x_0$ to a point of $x\in D$. Then $Im(\sigma)$ meets each $C_n$ at a unique 
point, say, $x_n=\sigma(t_n)$.  Let $\sigma_n:I\to \mathring \Sigma$ be the straight line segment joining $x_0$ to $x_n$. 
Let $\xi_n:I\to \mathring \Sigma$ be 
a Jordan curve based at $x_n$ with image $C_n$.  Finally, define $\zeta_n:=q\circ (\sigma_n\cdot \xi_n\cdot \sigma_n^{-1}).$  Then $\zeta_n$ is a loop in $\mathring\Sigma/\mathring D_0$ based at $y_0$ and $\lim_{n\to\infty}(Im(\zeta_n))=y_0$. 
Therefore, the infinite product $[\zeta]:=\prod_{n\in \mathbb N}[\zeta_n]\in \pi_1(\mathring\Sigma/\mathring D_0,y_0)$ 
is well-defined. 
It is readily seen that $[\zeta]$ is not in the image of $q_*:\pi_1(\mathring\Sigma,x_0)\to \pi_1(\mathring\Sigma/D_0,y_0).$
\end{remark}

\begin{theorem}\label{N-nonHopf}
The group $\pi_1(\mathcal N)$ is not Hopfian.
\end{theorem}
\begin{proof}
We shall consider $\mathbb R^2$ as $\mathbb C$ so that $\mathcal N=\mathbb C\setminus \mathbb Q[i]$. 
Let 
$z=x+iy, x,y\in \mathbb R$.
Define a map $ F : \mathbb{C} \longrightarrow \mathbb{C}$ as follows: 
\[ F(z)=\begin{cases}
    z, & \textrm{~if~}x \leq 0,\\
    -x+iy, & \textrm{~if~} 0 \leq x \leq 1,\\
    z-2, & \textrm{~if~}x \geq 1.\\
\end{cases}
\]
Then $F^{-1}(\mathbb Q[i])=\mathbb Q[i]$ and so it restricts to a surjective continuous map 
$f: \mathcal N\to \mathcal N$.  We shall use the point $z_0:=i\pi\in \mathcal N$ as the base point. 

It is clear that $f_*:\pi_1(\mathcal N, z_0)\to \pi_1(\mathcal N,z_0)$ is 
not a monomorphism.  For example, $f\circ \alpha$ is null-homotopic where $\alpha:I\to 
\mathcal N$ is the loop $t\mapsto i(\pi+\frac{1}{2})+\frac{1}{2}e^{2\pi i t}$.
It remains to show that $f_*:\pi_1(\mathcal N)\to \pi_1(\mathcal N)$ is 
surjective. 

Set $U_0:=\{z\in \mathcal N\mid x<0\}, U_1=\{z\in\mathcal N\mid x>1\},  
U_2=\{z\in \mathcal N\mid x>2\}, 
$ where $z=x+iy$.  Then $V_0:=f(U_0)=U_0,V_1:=f(U_1)=\{z\in \mathcal N\mid x>-1\}, V_2:=f(U_2)=\{z\in \mathcal N\mid x>0\}$. Moreover, $f$ maps $U_j$ to $V_j$ homeomorphically when $j=0,1,2$.  
Set $K:=\{z\in \mathcal N\mid -1\le x\le 0\}$.

  Suppose that $\gamma:I\to \mathcal N$ is a reduced loop (based at $z_0$). 
  We must find a loop $\widetilde \gamma $ based at $z_0$ such that $f_*([\widetilde \gamma])=[\gamma]$. There are several cases to consider depending upon the location of $Im(\gamma)$ in $\mathcal N.$ 

{\em Case 1:}  Suppose that $Im(\gamma)\subset \overline U_0.$  In this case 
$f\circ \gamma=\gamma$ and so $f_*([\gamma])=[\gamma].$

{\em Case 2:}  Suppose that $Im(\gamma)\subset \overline V_1$.
Choose $\sigma:I\to \mathcal N$ to be the (straight) line segment joining $z_0$
 to $z_1:=2+i\pi$ and let $\gamma_1:I\to \mathcal N$ be the loop $t\mapsto 
 2+\gamma(t).$  Note that $f(z_1)=z_0$ and that $f\circ \sigma$ is a null-homotopic loop based at $z_0$.  Also $f\circ \gamma_1=\gamma.$  It follows that $f_*([\sigma\cdot \gamma_1\cdot \sigma^{-1}])=[\gamma]\in \pi_1(\mathcal N,z_0).$

 {\em Case 3:}  Suppose that 
 $Im(\gamma)\cap (\mathcal N\setminus \overline U_0)\ne 
 \emptyset\ne Im(\gamma)\cap (\mathcal N\setminus \overline V_1) $.  Thus 
 $Im(\gamma)$ meets {\em both} the components of $\mathcal N\setminus K$. 
 Let $Y_r=\{z\in \mathcal N\mid x=r\}, r\in \mathbb R$.\\
{\em Claim:}
There is a {\em finite} non-empty collection $\mathcal J=\{J_r=[a_r,b_r]\}_{1\le r\le n}$ of closed subintervals of $I$ with $b_j<a_{j+1}$ for $1\le j<n$, such that the following holds for each $r$: \\
(i) $\gamma(a_r),\gamma(b_r)\in Y_0$,\\ 
(ii) $Re(\gamma(t))>-1$ for $a_r< t<b_r$, and, $Re(\gamma(s))=1$ for some $a_r<s<b_r$,\\
(iii) $J_r$ is maximal with respect to the above two properties, and, \\
(iv) the collection $\mathcal J$ is maximal. \\ 
~\\
Set $b_0:=0, a_{n+1}:=1$.  Note that, if $1\le r\le n$, then, by the maximality of $J_r$, and that of $\mathcal J$,
for $1\le r\le n$, there exists a $c_r\in (b_r,a_{r+1})$ such that    $Re(\gamma(c_r))=-1$,
and, \[Re(\gamma(t))<1 ~\mathrm{for~} b_r<t\le a_{r+1}.\]


Since $Im(\gamma)$ meets both the components of $\mathcal N\setminus K$ implies that $\mathcal J$ is non vacuous. The finiteness of $\mathcal J$ follows 
from the continuity of $\gamma$, in view of (i) and (ii). 

For $1\le r\le n$, define $\sigma_r,\tau_r,\gamma_r':J_r\to \mathcal N$ as $\sigma_r(t)=\gamma(a_r)+\frac{2(t-a_r)}{b_r-a_r}, \tau_r(t)=\gamma(b_r)+\frac{2(t-a_r)}{b_r-a_r}$, and, $\gamma_r'(t)=\gamma(t)+2.$  We set $\gamma_r:J_r\to \mathcal N$ to be the concatenation $\sigma_r\cdot \gamma'_r\cdot \tau_r^{-1}$.
Then $\gamma_r$ is continuous.
Note that $f\circ \gamma_r\simeq \gamma|_{J_r}$ relative to $\{a_r,b_r\}$ since 
$f\circ \sigma_r,f\circ \tau_r$ are null-homotopic and $f\circ \gamma'_r=\gamma|_{J_r}.$ 

Denote by $I_r$ the interval $[b_r,a_{r+1}], 0\le r\le n$.  Then $Re(\gamma(t))<1$ for all $t\in I_r.$
We shall prove, in Lemma \ref{image-left-of-1} below, that, for each $r$, there exists a 
path $\xi_r:I_r\to \mathcal N$ such that 
\[f\circ \xi_r\simeq \gamma|_{I_r} \mathrm{~relative~to~} \{b_r,a_{r+1}\}.\]

Define $\widetilde\gamma:I\to \mathcal N$ as follows:
\[\widetilde\gamma(t)=\begin{cases}
    \xi_s(t),&~\mathrm{~if~} t\in I_s, 0\le s\le n,\\
\gamma_r(t), &~\mathrm{if~} t\in J_r, 1\le r\le n.
\end{cases}\]
Then $\widetilde \gamma$ is continuous.  


Moreover, $f\circ\widetilde \gamma\simeq\gamma$ relative to $\{0,1\},$
since $f\circ \xi_s\simeq \gamma|_{I_s},0\le s\le n,$ and, 
$f\circ \gamma_r\simeq \gamma|_{J_r}$ for $1\le r\le n.$ This completes the proof.
\end{proof}

\begin{lemma}\label{image-left-of-1}
Suppose that $\xi:I\to \mathcal N$ is a path such that $\xi(0),\xi(1)\in Y_0$ and 
$Re(\xi(t))<1~\forall t\in I.$   Then there exists a path $\xi':I\to \mathcal N$ such that $\xi\simeq f\circ \xi'\mathrm{~relative~to~} \{0,1\}$.
\end{lemma}
\begin{proof}  By the continuity of $\xi$ there are only finitely many 
pairwise disjoint closed subintervals $I_1,\ldots, I_k\subset I$ such that, writing $I_j=[a_j,b_j], b_j<a_{j+1}, 1\le j<k,$ we have\\ (i) $\xi(a_j),\xi(b_j)\in Y_0$ with $\xi(c)\in Y_{-1}$ for some $c\in I_j,$ and, $Re(\xi(t))<0$ for $a_j<t<b_j$, \\ 
(ii)
$\xi(t)>-1$ for $t\in I\setminus\bigcup_{1\le j\le k} I_j.$

Let $J_i=[b_i, a_{i+1}], 0\le i\le k,$ where $b_0=0, a_{k+1}=1.$ Then $-1<Re(\xi(t))<1$ for all $t\in J_i.$  Define $\xi'_i:J_i\to \mathcal N$ as $t\mapsto 2+\xi(t).$
Let $\sigma_i,\tau_i:J_i\to \mathcal N$ be defined as $\sigma_i(t)=\xi(b_i)+\frac{2(t-b_i)}{a_{i+1}-b_i}, \tau_i(t)=\xi(a_{i+1})+\frac{2(t-b_i)}{a_{i+1}-b_i}$.  Then, 
$\xi_i:J_i\to \mathcal N$ defined as $\sigma_i\cdot\xi'_i\cdot \tau_i^{-1}$ is a well-defined continuous path such that $f\circ \xi_i\simeq \xi|_{J_i}$ relative to $\partial J_i$, for all $0\le i\le k$.  Note that $Re(\xi'_i(J_i))\subset(1,3).$
Finally, define 
\[\xi'(t)=\begin{cases} \xi(t), &~\mathrm{if~} t\in \bigcup_{1\le j\le k} I_j\\
\xi_i(t),& ~\mathrm{if~} t\in J_i, 0\le i\le k.\\
\end{cases}
\]
Then $f\circ \xi'\simeq \xi$ relative to $\{0,1\}$.
\end{proof}

The following corollary shows, for example, that $\pi_1(\mathring \Sigma)$ is not Hopfian for any connected surface $\Sigma.$ 
\begin{corollary}
    Suppose that $X$ is a path connected metrizable one-dimensional space which contains an open set homeomorphic to $\mathcal N.$  Then $\pi_1(X)$ is not Hopfian.
\end{corollary}

\begin{proof}
The idea of the proof is to find a self-map of $X$ 
which restricts to $f$ as in the above theorem on an open subspace $U\cong \mathcal N$ of $X$ and is identity outside of $U$. 

Let $K=[0,2]\times [-1,2]\subset \mathbb R^2$ and let $C=[0,1)\times (0,1)$.  
Let $Y'=\mathbb R^2\setminus K$ and $Y=Y'\cup C\subset \mathbb R^2$. 
Note $ Y\cong \mathbb R^2\setminus\{0\}.$ It follows that removal of a countable 
dense subset $A$ from $Y$ results in a space---denoted $\mathring Y$---which is  homeomorphic to $\mathcal N.$  
Also  $U_1:=(0,1)\times (0,1)\setminus A\subset \mathring Y$ is homeomorphic to $\mathcal N$. 
It is easily seen, using the above theorem, that $U_1$ admits a continuous map  $f_1:U_1\to U_1$ fixing a base point $x_1\in U_1$ such that:\\ (i) 
$f_{1,*}:\pi_1(U_1,x_1)\to\pi_1(U_1,x_1)$ is a surjection with nontrivial kernel, and,\\ (ii) $f_1$ restricts to the identity outside of $V_1:=(1/2,1)\times (0,1)\cap U_1.$ \\
It follows that $f_1$ extends to a continuous map $f:\mathring Y\to \mathring Y$
such that $f$ restricts to $f_1$ on $U_1$ and to the identity outside of $V_1$.  
In particular 
then $f_*:\pi_1(\mathring Y,x_1)\to \pi_1(\mathring Y,x_1)$ is a surjection with nontrivial kernel.
\begin{center}
\begin{tikzpicture}
    \fill[gray!50!white] (-1,-1.5) rectangle (2.5,2.5);
    \fill[white!50] (0,-0.5) rectangle (1.6,1.5);
    \fill [gray!50!white] (0,0) rectangle (1,1);
    \node at (2,2) {$\mathring Y$};
    \node at (1.2,1.2) {$K$};
    \node at (0.5,0.5) {$U_1$};
\end{tikzpicture}\\
 Figure 3
\end{center}
Now let $U\subset X$ be an open subset of $X$ homeomorphic to $\mathcal N$ such that $\partial_XU\cong \mathbb S^1.$  
Let $h:U\to \mathring Y$ be a homeomorphism.  We choose $x_0\in U$ such that $h(x_0)=x_1\in U_1\subset \mathring Y$.  The continuous 
map $\varphi: U\to U$ defined as $ h^{-1}\circ f \circ h$ induces a surjection $\varphi_*: \pi_1(U,x_0)\to \pi_1(U,x_0)$
with nontrivial kernel and further satisfying the following requirements: (i) $x_0\in V\subset \overline  V\subset U$ for an open set $V$ with $U\setminus \overline V$ path connected, (ii) 
$\varphi(\overline V)=\overline V$,  
and, (iii) $\varphi$ restricts to the identity on $U\setminus \overline V$. Therefore, we have a continuous extension $\widetilde \varphi:X\to X$ of $\varphi$ which is identity outside of $V$.  
Since $\pi_1(X,x_0)$ is generated by $\pi_1(U,x_0)$ and $\pi_1(X\setminus \overline V,x_0)$ (by 
Seifert-Van Kampen theorem), 
it follows that 
$\widetilde \varphi_*:\pi_1(X,x_0)\to \pi_1(X,x_0)$ is a surjection.
Clearly, 
$\ker (\varphi_*)\subset \ker(\widetilde \varphi_*)$ is nontrivial.  This completes the proof.
\end{proof}

\begin{lemma} \label{retract-non-Hopf}
Let $X$ be a path connected one-dimensional Hausdorff topological space. Let $X_0\subset X$ and suppose that $r:X\to X_0$ be a retraction.   
Suppose that $X$ is homeomorphic to $X_0$ and that $r_*:\pi_1(X,x_0)\to \pi_1(X_0,x_0)$ is not an isomorphism. 
Then $\pi_1(X,x_0)$ is not Hopfian.
\end{lemma}
\begin{proof}
    Since $r:X\to X_0$ is a retraction, $r_*:\pi_1(X,x_0)\to \pi_1(X_0,x_0)$ is a surjection. Fix a homeomorphism $\phi:(X_0,x_0)\to (X,x_1)$ and a path $\sigma$ in $X$ joining $x_0$ to $x_1$.  Then $\phi_*\circ r_*:\pi_1(X,x_0)\to \pi_1(X,x_1)$ is a surjection.   
    By our hypothesis on $r$ we see that $\ker(r_*)\ne \{1\}$ and so the same is true of $\phi_*\circ r_*$.  Composing $\phi_*\circ r_*$ with 
    the isomorphism $\sigma_*: \pi_1(X,x_1)\to \pi_1(X,x_0)$ that sends 
    $[\xi]$ to $[\sigma \cdot \xi\cdot \sigma^{-1}]$, we obtain a surjective homomorphism of $\pi_1(X,x_0)$ which is not an isomorphism.  This  
    completes the proof.
\end{proof}

The simplest examples of spaces satisfying the hypotheses 
of the theorem are the Hawaiian earring $\mathcal H_\aleph$ 
for any infinite cardinal $\aleph.$   We shall apply the above lemma to show that 
the fundamental groups of the Sierpiński curve $\mathcal S$, the Sierpiński gasket $\mathcal T$ and the Menger curve $\mathcal M$ are non-Hopfian.  

\begin{theorem}
    The groups $\pi_1(\mathcal S), \pi_1(\mathcal T)$ and $\pi_1(\mathcal M)$ are non-Hopfian.
\end{theorem}
\begin{proof} We begin by recalling the construction of $\mathcal T$ 
as the intersection $\bigcap_{n\ge 1} T_n$, 
starting with a triangle (i.e., a 
$2$-simplex) $T_1\subset \mathbb R^2$.  Each $T_n$ is a union of $3^{n-1}$ triangles and $T_{n+1}\subset T_{n}$.  The space $T_2$ is obtained from $T_1$ by first subdividing the $2$-simplex $T_1$ into four congruent triangles, by taking the mid-points of the edges of $T_1$ as {\em new} vertices, and then deleting the {\em interior} of the triangle whose vertices are all new.   Thus $T_2\subset T_1$ is a union of three triangles, each of which has exact two new vertices of $T_1$.  Having constructed $T_n$ as a union of $3^{n-1}$ triangles, $T_{n+1}$ is obtained by 
iterating the above construction to each of the $3^{n-1}$ triangles, resulting in $3^{n}$ pairwise congruent triangles.

Let $u_0,u_1,u_2\in \mathbb R^2$ be the vertices of $T=T_1$ and let $v_0,v_1,v_2$
be the mid-points of the edges $[u_1,u_2],[u_0,u_2], [u_0,u_1]$ of $T$ respectively.  Then the six vertices $u_i,v_i, 1\le i\le 3,$ yield a subdivision $\widetilde T$ of $T$ having four $2$-simplices. We obtain a simplicial map 
$\rho:\widetilde T\to \widetilde T$ defined by 
$\rho(u_0)=\rho(v_0)=u_0, \rho(v_1)=  \rho(u_1)=v_1,\rho(v_2)=\rho(u_2)=v_2.$ Then, denoting by $T_0$ the $2$-simplex with vertices $u_0,v_1,v_2$, $\rho$ defines a 
retraction $\rho_0:T\to T_0$.  Restricted to each of the 
$2$-simplices of $\widetilde T$, $\rho_0$ is an isometry.
Therefore $\rho_0$ defines a retraction $r:\mathcal T\to \mathcal T_0$ where $\mathcal T_0=\mathcal T\cap T_0.$ Note that $\mathcal T_0\cong \mathcal T.$  Hence $r_*:\pi_1(\mathcal T, v_2)\to \pi_1(\mathcal T_0,v_2)$ is a surjection.  

The positively oriented boundary of the triangle with vertices $u_1, v_0,v_2$, regarded as a loop $\gamma_1$
based at $v_2$ maps to $\gamma_0$ the positively oriented boundary of $T_0$ based at $v_2$.  Therefore $r_*[\gamma_1\cdot \gamma_0^{-1}]=1.$  Evidently $\gamma_1\cdot \gamma_0^{-1}$ is not null-homotopic.  By Lemma 
\ref{retract-non-Hopf} we conclude that $\pi_1(\mathcal T)$
is not Hopfian.

Next we consider the (standard) Sierpiński curve $\mathcal S\subset I^2.$  Let $J_0=[0,\frac{1}{3}], J_1=[\frac{1}{3},\frac{2}{3}], J_2=[\frac{2}{3},1]\subset I$ and set $\mathcal S_0:=\mathcal S\cap J_0^2$. 

Let $\rho:I\to I$ be the map 
\[
\rho(t)=\begin{cases}
    t, & \mathrm{if~} t\in J_0,\\
    \frac{2}{3}-t & \mathrm{if~} t\in J_1,\\
    t-\frac{2}{3} & \mathrm{if~} t\in J_2,
\end{cases}
\]
Then $\rho\times \rho:I^2\to I^2$ 
defines a retract $\rho_2: I^2\to J_0^2.$  It is readily verified that $\rho_2(\mathcal S)= \mathcal S\cap J_0^2$ and defines a retract    
$r: \mathcal S\to \mathcal S_0$ where $\mathcal S_0:=\mathcal S\cap J_0^2.$  Let $\sigma$ be the (positively oriented) 
boundary of $J_0\times [0,\frac{2}{3}]$ viewed as a loop in $\mathcal S$ based at $x_0=(0,0)\in \mathcal S.$ Then $r_*([\sigma])=1.$
Since $\mathcal S_0\cong \mathcal S$, by Lemma \ref{retract-non-Hopf}, we conclude that $\pi_1(\mathcal S)$ is not Hopfian.

The proof for $\mathcal M\subset I^3$ is analogous to the case of $\mathcal S$.  With $\rho:I\to I$ as above $\rho\times \rho\times \rho:I^3\to I^3$ yields a retract 
$\rho_3:I^3\to J_0^3$.  From the very construction of $\mathcal M$, it is easily verified that $\rho_3$ 
yields a retract  $r:\mathcal M\to \mathcal M_0$ where $\mathcal M_0=\mathcal M\cap J_0^3\cong \mathcal M.$ As in the case of Sierpiński curve, $r_*:\pi_1(\mathcal M)\to 
\pi_1(\mathcal M_0)$ is not a monomorphism and we conclude  
that $\pi_1(\mathcal M)$ is not Hopfian.
This completes the proof.
\end{proof}



\begin{thebibliography}{GG3}

\bibitem[AS]{ahlfors-sario} Ahlfors, L. V.; Sario, L. {\em Riemann Surfaces}.
Princeton Univ. Press, Princeton, NJ, (1960).
\bibitem[B]{bennett} Bennett, R. Countable dense homogeneous spaces. Fund. Math. {\bf 74} (1972) 189--194.
\bibitem[CC]{cannon-conner} Cannon, J. W.; Conner, G. R.  On the fundamental groups of one-dimensional spaces. Topology  Appl. {\bf 153} (2006) 2648--2672. 
\bibitem[CCZ]{cannon-conner-zastrow} Cannon, J. W.;  Conner G. R.; and Zastrow, A. One-dimensional sets and planar
sets are aspherical, Topology Appl. {\bf 120} (2002) 23--45.

\bibitem[CL]{conner-lamoreaux} Conner, G. R.; Lamoreaux, J. W. 
On the existence of universal covering spaces for
metric spaces and subsets of the Euclidean plane. Fund. Math. {\bf 187} (2005), 95--110.
\bibitem[CF1]{curtis-fort-1957} Curtis, M. L.; Fort, M. K. Homotopy groups of one-dimensional spaces. Proc. Amer. Math. Soc. {\bf 8} (1957) 577--579. 
\bibitem[CF2]{curtis-fort} Curtis, M. L.; Fort, M. K. Singular homology of one-dimensional spaces. Ann. Math. {\bf 69} (1959) 303--313. 


\bibitem[E1]{eda-ja} Eda, K. Free $\sigma$-products and noncommutatively slender groups. J. Alg. {\bf 148} (1992) 243--263. 
\bibitem[E2]{eda-2002} Eda, K. The fundamental groups of one-dimensional spaces and spacial homomorphisms. Topology and its Applications {\bf 123} (2002) 479--505. 
              
\bibitem[EK1]{eda-kawamura} Eda, K; Kawamura, K. The fundamental groups of one-dimensional spaces, Topology Appl. {\bf 87} (1998) 163--172.





\bibitem[G]{griffiths} Griffiths, H. B.
Infinite products of semigroups and local connectivity, Proc. London
Math. Soc. {\bf 6} (1954) 455--485.

\bibitem[H]{higman} Higman, G.  Unrestricted free products and varieties of topological groups, J. London
Math. Soc. {\bf 27} (1952) 73--81.

\bibitem[K]{kerekjarto} Kerékjártó, B.  {\em Vorlesungen über Topologie.} Springer, Berlin, 1923.
\bibitem[MM]{morgan-morrison} Morgan, J.; Morrison, I. A. A Van Kampen theorem for weak joins. Proc. London Math. Soc. (3), {\bf 53} (1986) 562--576.


\bibitem[Re]{reichbach} Reichbach, M. 
 The power of topological types of some classes of  $0$-dimensional sets.
Proc. Amer. Math. Soc. {\bf 13} (1962) 17--23.

\bibitem[Ri]{richards} Richards, I. On the classification of noncompact surfaces.  Trans. Amer.
Math. Soc. {\bf 106} (1963) 259--269. 



\bibitem[S]{spanier}  Spanier, E. H. {\em Algebraic topology.} Springer-Verlag New York. Originally published by Mac-Graw Hill, 1966.


\end{thebibliography}
\end{document}